\documentclass[leqno,12pt]{article}
\usepackage{hyperref}

\usepackage{amsmath,amssymb}
\usepackage{amsthm}

\def\address#1#2{\begingroup
\noindent\parbox[t]{9cm}{%
\small{\scshape\ignorespaces#1}\par\vskip1ex
\noindent\small{\itshape E-mail address}%
\/: #2\par\vskip4ex}\hfill%
\endgroup}%

\theoremstyle{plain}
\newtheorem{theorem}{\indent\sc Theorem}[section]
\newtheorem{lemma}[theorem]{\indent\sc Lemma}
\newtheorem{corollary}[theorem]{\indent\sc Corollary}
\newtheorem{proposition}[theorem]{\indent\sc Proposition}
\newtheorem{conjecture}[theorem]{\indent\sc Conjecture}

\newtheorem{maintheorem}{Theorem}

\theoremstyle{definition}
\newtheorem{definition}[theorem]{\indent\sc Definition}
\newtheorem{remark}[theorem]{\indent\sc Remark}

\title{\uppercase{Weil-\'{e}tale cohomology and duality for arithmetic schemes
    in negative weights}}

\author{\textsc{Alexey Beshenov}}
\date{}

\usepackage{pdflscape} 

\DeclareMathOperator{\Spec}{Spec}
\DeclareMathOperator{\Gal}{Gal}
\DeclareMathOperator{\Hom}{Hom}

\DeclareMathOperator{\Ext}{Ext}
\DeclareMathOperator{\Tot}{Tot}
\DeclareMathOperator{\coker}{coker}

\newcommand{\CC}{\mathbb{C}}
\newcommand{\FF}{\mathbb{F}}

\newcommand{\QQ}{\mathbb{Q}}
\newcommand{\RR}{\mathbb{R}}
\newcommand{\ZZ}{\mathbb{Z}}

\renewcommand{\AA}{\mathbb{A}}
\renewcommand{\div}{\text{\it div}}
\newcommand{\tor}{\text{\it tor}}
\newcommand{\cotor}{\text{\it cotor}}

\newcommand{\dfn}{\mathrel{\mathop:}=}
\newcommand{\rdfn}{=\mathrel{\mathop:}}

\newcommand{\Wc}{\text{\it W,c}}
\newcommand{\ar}{\text{\it ar}}
\newcommand{\et}{\text{\it \'{e}t}}
\newcommand{\fg}{\text{\it fg}}

\newcommand{\iHom}{\underline{\Hom}}
\newcommand{\RHom}{R\!\Hom}

\newcommand{\DZ}{{\mathbf{D} (\ZZ)}}

\usepackage{tikz-cd}
\usetikzlibrary{arrows}
\usetikzlibrary{calc}
\usetikzlibrary{babel}
\usetikzlibrary{decorations.pathreplacing}
\usetikzlibrary{patterns}

\newcommand{\tikzpb}{\ar[phantom,pos=0.2]{dr}{\text{\large$\lrcorner$}}}
\newcommand{\tikzpbur}{\ar[phantom,pos=0.2]{dl}{\text{\large$\llcorner$}}}

\begin{document}

\maketitle



\begin{abstract}
  Flach and Morin (2018) constructed Weil-\'{e}tale cohomology
  $H^i_\Wc (X, \ZZ(n))$ for a proper, regular arithmetic scheme $X$
  (i.e. separated and of finite type over $\Spec \ZZ$) and $n \in \ZZ$.
  In the case when $n < 0$, we generalize their construction to
  an arbitrary arithmetic scheme $X$, thus removing the proper
  and regular assumption. The construction uses \'{e}tale motivic cohomology
  groups $H^i (X_\et, \ZZ^c (n))$, as studied by Geisser (2010),
  and assumes their finite generation for $n < 0$. We give a class of $X$
  for which finite generation is known, and hence $H^i_\Wc (X, \ZZ(n))$
  is defined unconditionally.
\end{abstract}

\section{Introduction}

Lichtenbaum, in a series of papers
\cite{Lichtenbaum-2005,Lichtenbaum-2009-Euler-char,Lichtenbaum-2009-number-rings},
has envisioned a new cohomology theory for schemes, known as
\textbf{Weil-\'{e}tale cohomology}. The case of varieties over finite fields
$X/\FF_q$ was further studied by Geisser
\cite{Geisser-2004,Geisser-2006,Geisser-2010-arithmetic-homology}. Morin
defined in \cite{Morin-2014} Weil-\'{e}tale cohomology with compact support
$H^i_\Wc (X, \ZZ)$ for $X \to \Spec\ZZ$ separated, of finite type, proper, and
regular. This construction was further generalized by Flach and Morin in
\cite{Flach-Morin-2018} to the groups $H^i_\Wc (X, \ZZ(n))$ with arbitrary
weights $n \in \ZZ$, under the same assumptions on $X$.

The aim of this paper is to remove the assumption that $X$ is proper and
regular and, following the ideas of \cite{Flach-Morin-2018}, to construct the
groups $H^i_\Wc (X, \ZZ(n))$ for any $X$ separated and of finite type over
$\Spec\ZZ$ for the case of strictly negative weights $n < 0$.

As Flach and Morin already suggest in \cite[Remark 3.11]{Flach-Morin-2018}, we
rework all their constructions in terms of cycle complexes $\ZZ^c (n)$, which
were considered by Geisser in \cite{Geisser-2010} in the context of arithmetic
duality theorems.

In a forthcoming paper we apply the results of this text to relate the
cohomology groups $H_\Wc^i (X, \ZZ(n))$ to the special value of the zeta
function $\zeta (X, s)$ at $s = n < 0$.

\subsection*{Notation and conventions}

\paragraph{Arithmetic schemes.}
In this work, an \textbf{arithmetic scheme} is a scheme $X$ that is separated
and of finite type over $\Spec \ZZ$.

\paragraph{Abelian groups.}
Let $A$ be an abelian group. For $m \ge 1$ we denote by ${}_m A$ its
$m$-torsion subgroup, and by $A_m$ the quotient $A/mA$: $$0 \to {}_m A \to A
  \xrightarrow{\times m} A \to A_m \to 0$$ We denote by $A_\div$ (resp. $A_\tor$)
the maximal divisible subgroup (resp. maximal torsion subgroup), and by
$A_\cotor$ the quotient $A/A_\tor$ (following the notation in
\cite{Flach-Morin-2018}).

We say that $A$ is of \textbf{cofinite type} if it is $\QQ/\ZZ$-dual to a
finitely generated abelian group: $A = \Hom (B,\QQ/\ZZ)$ for a finitely
generated $B$.

\paragraph{Complexes.}
All our constructions take place in the derived category of abelian groups
$\DZ$. For our purposes, we introduce the following terminology. Recall first
that a complex of abelian groups $A^\bullet$ is \textbf{perfect} if it is
bounded (i.e. $H^i (A^\bullet) = 0$ for $|i| \gg 0$), and $H^i (A^\bullet)$ are
finitely generated abelian groups.

\begin{definition}
  \label{dfn:almost-of-(co)finite-type}
  A complex of abelian groups $A^\bullet$ is \textbf{almost perfect}
  if the cohomology groups $H^i (A^\bullet)$ are finitely generated, and
  bounded, except for possible finite $2$-torsion in arbitrarily high degree.
  That is, $H^i (A^\bullet) = 0$ for $i \ll 0$ and $H^i (A^\bullet)$ is finite
  $2$-torsion for $i \gg 0$.

  A complex of abelian groups $A^\bullet$ is of \textbf{cofinite type} if the
  cohomology groups $H^i (A^\bullet)$ are of cofinite type and bounded.

  A complex of abelian groups $A^\bullet$ is \textbf{almost of cofinite type} if
  the cohomology groups $H^i (A^\bullet)$ are of cofinite type and bounded,
  except for possible finite $2$-torsion in arbitrarily high degree.
\end{definition}

This terminology is ad hoc and was invented for this text, since such complexes
will appear frequently. Some basic observations about almost perfect and almost
cofinite type complexes are collected in
Appendix~\ref{app:homological-algebra}. We note that this finite $2$-torsion in
arbitrarily high degrees could be removed by working with the Artin--Verdier
topology $\overline{X}_\et$ instead of the usual \'{e}tale topology $X_\et$.
The general construction and basic properties of $\overline{X}_\et$ are treated
in \cite[Appendix~A]{Flach-Morin-2018}, but only for a \emph{proper and
  regular} arithmetic scheme $X$. Our methods circumvent this restriction at the
cost of some technical hurdles with $2$-torsion.

\paragraph{\'{E}tale cohomology.} For an arithmetic scheme $X$ and a complex of \'{e}tale
sheaves $\mathcal{F}^\bullet$, we denote by
\[ R\Gamma (X_\et, \mathcal{F}^\bullet) ~
  \text{(resp. }R\Gamma_c (X_\et, \mathcal{F}^\bullet), ~
  R\widehat{\Gamma}_c (X_\et, \mathcal{F}^\bullet)\text{)} \]
the complex that computes the corresponding cohomology, resp. cohomology with
compact support, and modified cohomology with compact support. For the
convenience of the reader, we review the definitions in
Appendix~\ref{app:modified-cohomology-with-compact-support}. The purpose of
$R\widehat{\Gamma}_c (X_\et, \mathcal{F}^\bullet)$ is to take care of real
places $X (\RR)$. There exists a canonical projection $R\widehat{\Gamma}_c
  (X_\et, \mathcal{F}^\bullet) \to R\Gamma_c (X_\et, \mathcal{F}^\bullet)$, which
is an isomorphism if $X (\RR) = \emptyset$.

\paragraph{$G$-equivariant sheaves and their cohomology.}
Let $\mathcal{X}$ be a topological space with an action of a discrete group
$G$. A \textbf{$G$-equivariant sheaf} $\mathcal{F}$ on $\mathcal{X}$ can be
defined as an espace \'{e}tal\'{e} $\pi\colon E\to \mathcal{X}$ with a
$G$-action on $E$ such that the projection $\pi$ is $G$-equivariant (see e.g.
\cite[\S II.6 + pp.\,594]{MacLane-Moerdijk}). We denote by $\mathbf{Sh} (G,
  \mathcal{X})$ the corresponding category.

The equivariant global sections are defined by $$\Gamma
  (G,\mathcal{X},\mathcal{F}) = \mathcal{F} (\mathcal{X})^G,$$ with $G$ acting on
$\mathcal{F} (\mathcal{X}) = \{ s\colon \mathcal{X}\to E \mid \pi\circ s =
  id_\mathcal{X} \}$ via $(g\cdot s) (x) = g\cdot s (g^{-1}\cdot x)$. The
corresponding \textbf{$G$-equivariant cohomology} is given by the right derived
functors of $\Gamma (G,\mathcal{X},-)$.

More details on $G$-equivariant sheaves can be found in
\cite[Chapitre~2]{Morin-these}. For our modest purposes, it suffices to know
that any $G$-module $A$ gives rise to the corresponding abelian $G$-equivariant
constant sheaf. The latter corresponds to the espace \'{e}tal\'{e}
$\mathcal{X}\times A \to \mathcal{X}$, where $A$ is endowed with the discrete
topology.

\paragraph{$G_\RR$-equivariant cohomology of $X (\CC)$.}
Given an arithmetic scheme $X$, we denote by $X (\CC)$ the set of complex
points of $X$ endowed with the analytic topology. It carries the natural action
of the Galois group $G_\RR \dfn \Gal (\CC/\RR)$.

We consider the $G_\RR$-modules
\[ \ZZ (n) \dfn (2\pi i)^n\,\ZZ, \quad
  \QQ (n) \dfn (2\pi i)^n\,\QQ, \quad
  \QQ/\ZZ (n) \dfn \QQ (n) / \ZZ (n) \]
as constant $G_\RR$-equivariant sheaves on $X (\CC)$.

Then $R\Gamma_c (X (\CC), A (n))$ for $A = \ZZ, \QQ, \QQ/\ZZ$ (the complex that
computes singular cohomology with compact support of $X (\CC)$ with
coefficients in $A (n)$) is a complex of $G_\RR$-modules, and we can further
take the group cohomology (resp. Tate cohomology):
\begin{align*}
  R\Gamma_c (G_\RR, X (\CC), A (n))           & \dfn R\Gamma (G_\RR, R\Gamma_c (X (\CC), A (n))),           \\
  R\widehat{\Gamma}_c (G_\RR, X (\CC), A (n)) & \dfn R\widehat{\Gamma} (G_\RR, R\Gamma_c (X (\CC), A (n))).
\end{align*}
By definition, this is the \textbf{$G_\RR$-equivariant cohomology}
(resp. \textbf{$G_\RR$-equivariant Tate cohomology})
\textbf{with compact support} of $X (\CC)$ with coefficients in $A (n)$.

\paragraph{Motivic cohomology $H^i (X_\et, \ZZ^c (n))$.}
Our construction is based on motivic cohomology defined in terms of complexes
of sheaves $\ZZ^c (n)$ on $X_\et$. We follow the notation of
\cite{Geisser-2010}.

Briefly, for $i \ge 0$ we consider the algebraic simplex $$\Delta^i = \Spec
  \ZZ[t_0,\ldots,t_i]/(\sum_i t_i - 1).$$ We fix a non-positive weight $n \le 0$.
Let $z_n (X,i)$ be the free abelian group generated by the closed integral
subschemes $Z \subset X \times \Delta^i$ of dimension $n + i$ that intersect
the faces properly. Then $z_n (X, \bullet)$ is a (homological) complex of
abelian groups whose differentials are given by the alternating sum of
intersections with the faces. We consider the (cohomological) complex of
\'{e}tale sheaves $$\ZZ^c (n) \dfn z_n (\text{\textvisiblespace}, -\bullet)
  [2n].$$

The boundedness from below of $\ZZ^c(n)$ is not known in general; it is a
variant of the Beilinson--Soul\'{e} vanishing conjecture. To work
unconditionally with the derived functors, we use $K$-injective resolutions
\cite{Spaltenstein-1988,Serpe-2003} (resp. $K$-flat resolutions for the derived
tensor products).

To avoid any confusion, we use cohomological numbering for all complexes in
this paper, so we set $$H^i (X_\et, \ZZ^c(n)) \dfn H^i (R\Gamma (X_\et,
  \ZZ^c(n))).$$ (\cite{Geisser-2010} uses homological numbering.)

\subsection*{Assumptions}

\paragraph{Weights.}
In this paper, $n < 0$ always denotes a strictly negative integer, which will
be the weight in the cohomology groups $H^i_\Wc (X, \ZZ(n))$.

\paragraph{Finite generation conjecture.}
Our construction of the Weil-\'{e}tale cohomology groups $H^i_\Wc (X,\ZZ(n))$
uses the following assumption.

\begin{conjecture}
  $\mathbf{L}^c (X_\et,n)$: for an arithmetic scheme $X$ and $n < 0$,
  the cohomology groups $H^i (X_\et, \ZZ^c (n))$ are finitely generated for all
  $i \in \ZZ$.
\end{conjecture}

See Proposition~\ref{prop:equivalent-conjectures} for consistency of
$\mathbf{L}^c (X_\et,n)$ with other conjectures that appear in the literature.
We refer to \S\ref{sec:known-cases-of-Lc-Xet-n} for the cases where the
conjecture is known.

\subsection*{Main results}

Here we state the main results of this paper that are needed for the
construction of Weil-\'{e}tale cohomology. One of our main objects is the
following complex of abelian sheaves $\ZZ (n)$ on $X_\et$.

\begin{definition}[{\cite[\S 3.1]{Flach-Morin-2018}}, {\cite[\S 7]{Geisser-2004}}]
  \label{dfn:sheaf-Z(n)}
  Let $X$ be an arithmetic scheme and $n < 0$. For a prime $p$, consider
  the localization $X [1/p]$, and let $\mu_{p^r}$ be the sheaf of $p^r$-th
  roots of unity on $X [1/p]$. We define the twist of $\mu_{p^r}$ by $n$ as
  $$\mu_{p^r}^{\otimes n} = \iHom_{X[1/p]} (\mu_{p^r}^{\otimes (-n)}, \ZZ/p^r\ZZ).$$

  Now $\ZZ (n)$ is the complex of sheaves on $X_\et$ given by
  \[ \ZZ (n) = \QQ/\ZZ (n) [-1],
    \quad \text{where }
    \QQ/\ZZ (n) = \bigoplus_p \varinjlim_r j_{p!} \mu_{p^r}^{\otimes n}, \]
  and $j_p$ is the canonical open immersion $X[1/p] \to X$.
\end{definition}

The above sheaves $\ZZ (n)$ should not be confused with cycle complexes; the
latter are $\ZZ^c (n)$ in the context of this paper. In
\S\ref{sec:arithmetic-duality-theorem} we prove the following arithmetic
duality theorem relating the two.

\begin{maintheorem}
  \label{theorem-I}
  Assuming Conjecture $\mathbf{L}^c (X_\et,n)$, for $n < 0$,
  there is a quasi-isomorphism
  \[ R\widehat{\Gamma}_c (X_\et, \ZZ (n)) \xrightarrow{\cong}
    \RHom (R\Gamma (X_\et, \ZZ^c (n)), \QQ/\ZZ [-2]). \]
\end{maintheorem}

The second result is related to the following morphism of complexes.

\begin{definition}
  \label{dfn:u-infty}
  We define
  $v_\infty^*\colon R\Gamma_c (X_\et, \QQ/\ZZ (n)) \to R\Gamma_c (G_\RR, X
    (\CC), \QQ/\ZZ (n))$ as the morphism in the derived category $\DZ$ induced by
  the comparison of \'{e}tale and analytic topology
  \[ \Gamma_c (X_\et, \QQ/\ZZ (n)) \to
    \Gamma_c (G_\RR, X (\CC), \alpha^* \QQ/\ZZ (n)) \cong
    \Gamma_c (G_\RR, X (\CC), \QQ/\ZZ (n)) \]
  (see Proposition~\ref{prop:inverse-image-gamma} and
  \ref{propn:image-of-Q/Zn-under-alpha}). Then we let
  $u_\infty^*\colon R\Gamma_c (X_\et, \ZZ(n)) \to R\Gamma_c (G_\RR, X (\CC), \ZZ (n))$
  be the composition
  \begin{multline*}
    R\Gamma_c (X_\et, \ZZ(n)) \dfn R\Gamma_c (X_\et, \QQ/\ZZ (n)) [-1]
    \xrightarrow{v_\infty^* [-1]} R\Gamma_c (G_\RR, X (\CC), \QQ/\ZZ (n)) [-1]
    \\ \to R\Gamma_c (G_\RR, X (\CC), \ZZ (n))
  \end{multline*}
  where the last arrow is induced by $\QQ/\ZZ (n) [-1] \to \ZZ (n)$, which comes
  from the distinguished triangle of constant $G_\RR$-equivariant sheaves
  $\ZZ (n) \to \QQ (n) \to \QQ/\ZZ (n) \to \ZZ (n) [1]$.
\end{definition}

\begin{maintheorem}
  \label{theorem-II}
  The morphism
  $u_\infty^*\colon R\Gamma_c (X_\et, \ZZ(n)) \to R\Gamma_c (G_\RR, X (\CC), \ZZ (n))$
  for $n < 0$ is torsion, i.e. there exists a nonzero integer $m$ such that
  $mu^*_\infty = 0$.
\end{maintheorem}

\subsection*{Outline of the paper}

Here we describe the structure of this paper, as well as our construction of
the Weil-\'{e}tale complexes $R\Gamma_\Wc (X, \ZZ (n))$.

First, \S\ref{sec:arithmetic-duality-theorem} is devoted to the proof of
Theorem~\ref{theorem-I}. Some of its consequences are deduced in
\S\ref{sec:consequences-of-theorem-I}. Namely, if we assume Conjecture
$\mathbf{L}^c (X_\et, n)$, then $R\Gamma (X_\et, \ZZ^c (n))$ is an almost
perfect complex, while $R\Gamma_c (X_\et, \ZZ (n))$ is almost of cofinite type
in the sense of Definition~\ref{dfn:almost-of-(co)finite-type}. For this, we
first make a small digression in \S\ref{sec:GR-equivariant-cohomology} to
analyze what kind of complexes we obtain for the $G_\RR$-equivariant cohomology
of $X (\CC)$.

Theorem~\ref{theorem-I} is used in \S\ref{sec:RGamma-fg} to define a morphism
$\alpha_{X,n}$ in the derived category (see Definition~\ref{def:RGamma-fg}),
and declare $R\Gamma_\fg (X, \ZZ(n))$ to be its cone:
\begin{multline*}
  \RHom (R\Gamma (X_\et, \ZZ^c (n)), \QQ [-2]) \xrightarrow{\alpha_{X,n}}
  R\Gamma_c (X_\et, \ZZ (n)) \to
  R\Gamma_\fg (X, \ZZ(n)) \\
  \to \RHom (R\Gamma (X_\et, \ZZ^c (n)), \QQ [-1])
\end{multline*}
The notation ``\emph{fg}'' comes from the fact that $R\Gamma_\fg (X, \ZZ(n))$ is
an almost perfect complex in the sense of
Definition~\ref{dfn:almost-of-(co)finite-type}. Thanks to specific properties of
the complexes involved, it turns out that $R\Gamma_\fg (X, \ZZ(n))$ is defined
up to a \emph{unique} isomorphism in the derived category (which is not normally
expected from a cone).

Then in \S\ref{sec:theorem-II} we establish Theorem~\ref{theorem-II}, and it is
used in \S\ref{sec:RGamma-Wc} to define Weil-\'{e}tale complexes $R\Gamma_\Wc
  (X, \ZZ(n))$. To do this, we deduce from Theorem~\ref{theorem-II} that
$u_\infty^* \circ \alpha_{X,n} = 0$, which implies that there exists a morphism
in the derived category $$i_\infty^*\colon R\Gamma_\fg (X, \ZZ (n)) \to
  R\Gamma_c (G_\RR, X (\CC), \ZZ(n)).$$ We choose a mapping fiber of $i_\infty^*$
and call it $R\Gamma_\Wc (X, \ZZ (n))$, which turns out to be a perfect
complex. The definition of $R\Gamma_\Wc (X,\ZZ(n))$ fits in the following
commutative diagram with distinguished triangles in the derived category $\DZ$:
\begin{equation*}
  \begin{tikzcd}[column sep=1em,font=\small]
    &[-3em] \RHom (R\Gamma (X_\et, \ZZ^c (n)), \QQ [-2]) \ar{d}{\alpha_{X,n}}[swap]{\text{Dfn.~\ref{def:RGamma-fg}}} \ar{r} &[-2.5em] 0 \ar{d} \\
    & R\Gamma_c (X_\et, \ZZ(n)) \ar{d}\ar{r}{u_\infty^*}[swap]{\text{Dfn.~\ref{dfn:u-infty}}} & R\Gamma_c (G_\RR, X (\CC), \ZZ(n))\ar{d}{id} \\
    R\Gamma_\Wc (X, \ZZ (n)) \ar{r} & R\Gamma_\fg (X, \ZZ(n)) \ar[dashed]{r}{i_\infty^*}\ar{d} & R\Gamma_c (G_\RR, X (\CC), \ZZ(n)) \ar{r} \ar{d} & R\Gamma_\Wc (X, \ZZ (n)) [1] \\
    & \RHom (R\Gamma (X_\et, \ZZ^c (n)), \QQ [-1]) \ar{r} & 0
  \end{tikzcd}
\end{equation*}

The resulting complex is the same as defined in \cite{Flach-Morin-2018} if $X$
is proper and regular.

In \S\ref{sec:known-cases-of-Lc-Xet-n} we consider the cases of $X$ for which
Conjecture $\mathbf{L}^c (X_\et, n)$ is known, and hence our results hold
unconditionally, and in \S\ref{sec:comparison-with-FM} we verify that if $X$ is
proper and regular, our complex $R\Gamma_\Wc (X, \ZZ (n))$ is isomorphic to
that constructed in \cite{Flach-Morin-2018} by Flach and Morin.

There are two appendices to this paper: Appendix~\ref{app:homological-algebra}
collects some lemmas from homological algebra, and
Appendix~\ref{app:modified-cohomology-with-compact-support} gives an overview
of the definitions of \'{e}tale cohomology with compact support $R\Gamma_c
  (X_\et, -)$ and its modified version $R\widehat{\Gamma}_c (X_\et, -)$.

This work is inspired by \cite{Flach-Morin-2018}. Here is a brief comparison
between the notation and assumptions.

\begin{center}
  \renewcommand{\arraystretch}{1.5}
  \begin{tabular}{cc}
    \hline
    \textbf{this paper}                                                                                           & \textbf{Flach--Morin}                                                                                                       \\
    \hline
    {\renewcommand{\arraystretch}{1}\begin{tabular}{c} $X\to\Spec\ZZ$ \\ separated, of finite type \end{tabular}} & {\renewcommand{\arraystretch}{1}\begin{tabular}{c} $X\to\Spec\ZZ$ \\ proper, regular, equidimensional\end{tabular}}         \\
    \hline
    $n < 0$                                                                                                       & $n \in \ZZ$                                                                                                                 \\
    \hline
    {\renewcommand{\arraystretch}{1}\begin{tabular}{c} cycle complexes \\ $\ZZ^c (n)$ \end{tabular}}              & {\renewcommand{\arraystretch}{1}\begin{tabular}{c} cycle complexes \\ $\ZZ (d-n)[2d]$, $d = \dim X$ \end{tabular}} \\
    \hline
    $R\Gamma_\fg (X, \ZZ(n))$                                                                                     & {\renewcommand{\arraystretch}{1}\begin{tabular}{c} $R\Gamma_W (\overline{X}, \ZZ(n))$, \\ up to finite $2$-torsion \end{tabular}}    \\
    \hline
    $R\Gamma_\Wc (X,\ZZ(n))$                                                                                      & $R\Gamma_\Wc (X, \ZZ(n))$                                                                                                   \\
    \hline
  \end{tabular}
\end{center}

{\small
\subsection*{Acknowledgments}

This text is based on the results of my PhD thesis, carried out at the
Universit\'{e} de Bordeaux and Universiteit Leiden under the supervision of
Baptiste Morin and Bas Edixhoven. I am very grateful to them for their support.
I thank Stephen Lichtenbaum and Niranjan Ramachandran who kindly agreed to act
as reviewers for my thesis and provided me with many useful comments and
suggestions. I am also indebted to Matthias Flach, since the ideas of this
paper come from \cite{Flach-Morin-2018}. Moreover, the work of Thomas Geisser
on arithmetic duality \cite{Geisser-2010} is also crucial for this paper, and
his work on Weil-\'{e}tale cohomology for varieties over finite fields
\cite{Geisser-2004,Geisser-2006,Geisser-2010-arithmetic-homology} has been of
great influence for me. I thank Maxim Mornev for many fruitful mathematical
conversations. This paper was edited during my stay at the Center for Research
in Mathematics (CIMAT), Guanajuato, Mexico. I~am grateful personally to Pedro
Luis del \'{A}ngel and Xavier G\'{o}mez Mont for their hospitality. Finally, I
am indebted to the anonymous referee whose sharp and insightful comments on an
earlier draft helped to improve the exposition.}


\section{Proof of Theorem~I}
\label{sec:arithmetic-duality-theorem}

At the heart of our constructions is an arithmetic duality theorem for cycle
complexes established by Thomas Geisser in \cite{Geisser-2010}. The purpose of
this section is to deduce Theorem~\ref{theorem-I} from Geisser's duality. We
would like to obtain a quasi-isomorphism of complexes
\[ R\widehat{\Gamma}_c (X_\et, \ZZ (n)) \xrightarrow{\cong}
  \RHom (R\Gamma (X_\et, \ZZ^c (n)), \QQ/\ZZ [-2]). \]

Here $R\widehat{\Gamma}_c (X_\et, \ZZ (n))$ denotes the modified \'{e}tale
cohomology with compact support, described in
Appendix~\ref{app:modified-cohomology-with-compact-support}. We note that
\cite{Geisser-2010} uses the notation ``$R\Gamma_c$'' for our
``$R\widehat{\Gamma}_c$'', but we take special care to distinguish the two
things, since we also need the usual \'{e}tale cohomology with compact support
$R\Gamma_c (X_\et, \ZZ (n))$.

We split our proof of Theorem~\ref{theorem-I} into two propositions.

\begin{proposition}
  For any $n < 0$ we have a quasi-isomorphism of complexes
  \begin{equation}
    \label{eqn:duality-quasi-isomorphism-1}
    R\widehat{\Gamma}_c (X_\et, \ZZ (n)) \cong
    \varinjlim_m \RHom (R\Gamma (X_\et, \ZZ/m\ZZ^c (n)), \QQ/\ZZ [-2]).
  \end{equation}

  \begin{proof}
    We unwind our definition of $\ZZ (n)$ for $n < 0$ and reduce everything to
    the results from \cite{Geisser-2010}. Since
    $\ZZ (n) \dfn \bigoplus_p \varinjlim_r j_{p!} \mu_{p^r}^{\otimes n} [-1]$,
    and \'{e}tale cohomology commutes with filtered colimits of coefficients,
    it suffices to show that for every prime $p$ and $r\ge 1$ there is a
    quasi-isomorphism of complexes
    \begin{equation}
      \label{eqn:duality-quasi-isomorphism-1-pr}
      R\widehat{\Gamma}_c (X_\et, j_{p!} \mu_{p^r}^{\otimes n} [-1]) \cong
      \RHom (R\Gamma (X_\et, \ZZ^c/p^r (n)), \QQ/\ZZ [-2]).
    \end{equation}

    As in Definition~\ref{dfn:sheaf-Z(n)}, here $j_p$ denotes the canonical open
    immersion $j_p\colon X[1/p] \hookrightarrow X$. We further denote by $f$ the
    structure morphism $X\to \Spec \ZZ$ and by $f_p$ the structure morphism $X
        [1/p] \to \Spec \ZZ [1/p]$:

    \[ \begin{tikzcd}
        X [1/p]\ar[hookrightarrow]{r}{j_p}\ar{d}[swap]{f_p} & X\ar{d}{f} \\
        \Spec \ZZ [1/p]\ar[hookrightarrow]{r} & \Spec \ZZ
      \end{tikzcd} \]

    As we are going to change the base scheme, let us write $\Hom_X (-,-)$ for the
    $\Hom$ between sheaves on $X_\et$ and $\iHom_X (-,-)$ for the internal $\Hom$.
    Instead of $\Hom_{\Spec R}$, we will simply write $\Hom_R$.

    Applying various results from \cite{Geisser-2004-Dedekind} and
    \cite{Geisser-2010}, we obtain a quasi-isomorphism of complexes of sheaves
    \[ R\iHom_X (j_{p!} \mu_{p^r}^{\otimes n} [-1], \ZZ^c_X (0)) \cong \hspace{4cm} \]
    \begin{align*}
       & \cong R j_{p*} R\iHom_{X [1/p]} (\mu_{p^r}^{\otimes n} [-1], \ZZ^c_{X [1/p]} (0))                                       & \text{\cite[Prop. 7.10~c)]{Geisser-2010}}          \\
       & \cong R j_{p*} R\iHom_{X[1/p]} (f_p^* \mu_{p^r}^{\otimes n} [-1], \ZZ^c_{X [1/p]} (0))                                                                                       \\
       & \cong R j_{p*} R f^!_p R\iHom_{\ZZ [1/p]} (\mu_{p^r}^{\otimes n} [-1], \ZZ^c_{\ZZ [1/p]} (0))                           & \text{\cite[Prop. 7.10~c)]{Geisser-2010}}          \\
       & \cong R j_{p*} R f^!_p R\iHom_{\ZZ [1/p]} (\mu_{p^r}^{\otimes n} [-1], \mathbb{G}_\mathrm{m} [1])                       & \text{\cite[Lemma~7.4]{Geisser-2010}}              \\
       & \cong R j_{p*} R f^!_p R\iHom_{\ZZ [1/p]} (\mu_{p^r}^{\otimes n}, \mathbb{G}_\mathrm{m}) [2]                                                                                 \\
       & \cong R j_{p*} R f^!_p \, \mu_{p^r}^{\otimes (1-n)} [2]                                                                                                                      \\
       & \cong R j_{p*} R f^!_p \, \Bigl(\ZZ_{\ZZ [1/p]}/p^r (1-n)\Bigr) [2] \cong R j_{p*} R f^!_p \, \ZZ^c_{\ZZ [1/p]}/p^r (n) & \text{\cite[Thm.~1.2]{Geisser-2004-Dedekind}}      \\
       & \cong R j_{p*} \ZZ^c_{X [1/p]} / p^r (n)                                                                                & \text{\cite[Prop.~7.10~a)]{Geisser-2010}}          \\
       & \cong R j_{p*} j_p^*\ZZ^c_X/p^r (n) \cong \ZZ^c_X/p^r (n)                                                               & \text{\cite[Thm.~7.2~a), Prop.~2.3]{Geisser-2010}}
    \end{align*}

    After applying $R\Gamma (X_\et, -)$, we get a quasi-isomorphism of complexes of
    abelian groups
    \[
      \RHom (j_{p!} \mu_{p^r}^{\otimes n} [-1], \ZZ^c_X (0)) \cong
      R\Gamma (X_\et, \ZZ^c_X/p^r (n)).
    \]

    Now according to the duality \cite[Theorem~7.8]{Geisser-2010},
    \[
      \RHom (j_{p !} \mu_{p^r}^{\otimes n} [-1], \ZZ^c (0)) \cong
      \RHom (R\widehat{\Gamma}_c (X_\et, j_{p !} \mu_{p^r}^{\otimes n} [-1]), \QQ/\ZZ [-2]).
    \]

    What we end up with is a quasi-isomorphism
    \[ R\Gamma (X_\et, \ZZ^c/p^r (n)) \cong \RHom (R\widehat{\Gamma}_c (X_\et,
      j_{p !} \mu_{p^r}^{\otimes n} [-1]), \QQ/\ZZ [-2]). \]
    The groups $\widehat{H}^i_c (X_\et, j_{p!} \mu_{p^r}^{\otimes n} [-1])$ are
    finite (the sheaves $j_{p!} \mu_{p^r}^{\otimes n}$ are constructible), so
    applying $\RHom (-,\QQ/\ZZ [-2])$ yields
    \eqref{eqn:duality-quasi-isomorphism-1-pr}.
  \end{proof}
\end{proposition}

To conclude the proof of Theorem~\ref{theorem-I}, we identify the complex on
the right-hand side of \eqref{eqn:duality-quasi-isomorphism-1}. For this, we
need Conjecture $\mathbf{L}^c (X_\et, n)$.

\begin{proposition}
  Assuming Conjecture $\mathbf{L}^c (X_\et, n)$, for $n < 0$, there is
  a quasi-isomorphism
  \[ \varinjlim_m \RHom (R\Gamma (X_\et, \ZZ/m\ZZ^c (n)), \QQ/\ZZ [-2]) \cong
    \RHom (R\Gamma (X_\et, \ZZ^c (n)), \QQ/\ZZ [-2]). \]

  \begin{proof}
    Consider short exact sequences
    \[ 0 \to H^i (X_\et, \ZZ^c (n))_m \to
      H^i (X_\et, \ZZ/m\ZZ^c (n)) \to
      {}_m H^{i+1} (X_\et, \ZZ^c (n)) \to 0 \]
    If we now take $\Hom (-,\QQ/\ZZ)$ and filtered colimits $\varinjlim_m$, we get
    \begin{multline}
      \label{eqn:short-exact-sequence-with-dirlim}
      0 \to \varinjlim_m \Hom ({}_m H^{i+1} (X_\et, \ZZ^c (n)), \QQ/\ZZ) \to \\
      \varinjlim_m \Hom (H^i (X_\et, \ZZ/m\ZZ^c (n)), \QQ/\ZZ) \to \\
      \varinjlim_m \Hom (H^i (X_\et, \ZZ^c (n))_m, \QQ/\ZZ) \to 0
    \end{multline}

    By Conjecture $\mathbf{L}^c (X_\et, n)$, the group $H^{i+1} (X_\et, \ZZ^c (n))$
    is finitely generated, and hence the first $\varinjlim_m$ in the short exact
    sequence \eqref{eqn:short-exact-sequence-with-dirlim} vanishes, and we obtain
    isomorphisms
    \[ \varinjlim_m \Hom (H^i (X_\et, \ZZ^c (n))_m, \QQ/\ZZ) \xrightarrow{\cong}
      \varinjlim_m \Hom (H^i (X_\et, \ZZ/m\ZZ^c (n)), \QQ/\ZZ). \]
    It remains to note that the left-hand side is canonically isomorphic to $\Hom
      (H^i (X_\et, \ZZ^c (n)), \QQ/\ZZ)$, again thanks to the finite generation of
    $H^i (X_\et, \ZZ^c (n))$, under Conjecture $\mathbf{L}^c (X_\et, n)$.

    To see this, observe that if $A$ is a finitely generated abelian group, there
    is a canonical isomorphism $$\varinjlim_m \Hom (A_m, \QQ/\ZZ) \cong \Hom (A,
      \QQ/\ZZ)$$ induced by $A \to A_m$, and then applying the functor $\Hom (-,
      \QQ/\ZZ)$ and $\varinjlim_m$. Since $\QQ/\ZZ$ is a torsion group, any
    homomorphism $A\to \QQ/\ZZ$ is killed by some $m$, hence factors through $A_m$.
  \end{proof}
\end{proposition}


\section{$G_\RR$-equivariant cohomology of $X (\CC)$}
\label{sec:GR-equivariant-cohomology}

\begin{lemma}
  Let $A^\bullet$ be a perfect complex of $\ZZ G_\RR$-modules.

  \begin{enumerate}
    \item[1)] The complex $A^\bullet \otimes^\mathbf{L} \QQ/\ZZ$ is of cofinite
          type.

    \item[2)]
          $R\Gamma (G_\RR, A^\bullet \otimes \QQ) \cong (A^\bullet \otimes
            \QQ)^{G_\RR}$ is a perfect complex of $\QQ$-vector spaces, and the complex
          $R\widehat{\Gamma} (G_\RR, A^\bullet \otimes \QQ)$ is quasi-isomorphic to
          $0$.

    \item[3)]
          $R\widehat{\Gamma} (G_\RR, A^\bullet \otimes^\mathbf{L} \QQ/\ZZ) \cong
            R\widehat{\Gamma} (G_\RR, A^\bullet [+1])$, and these complexes have finite
          $2$-torsion cohomology.

    \item[4)] $R\Gamma (G_\RR, A^\bullet)$ is almost perfect, and
          $R\Gamma (G_\RR, A^\bullet \otimes^\mathbf{L} \QQ/\ZZ)$ is almost of
          cofinite type.
  \end{enumerate}

  \begin{proof}
    The universal coefficient theorem gives us short exact sequences
    $$0 \to H^i (A^\bullet)_m \to H^i (A^\bullet \otimes^\mathbf{L} \ZZ/m\ZZ) \to {}_m H^{i+1} (A^\bullet) \to 0$$
    The colimit of these over $m$ is
    $$0 \to H^i (A^\bullet) \otimes \QQ/\ZZ \to H^i (A^\bullet \otimes^\mathbf{L} \QQ/\ZZ) \to H^{i+1} (A^\bullet)_\tor \to 0$$
    Here $H^i (A^\bullet) \otimes \QQ/\ZZ$ is injective, hence the short exact
    sequence splits. We see that $H^i (A^\bullet \otimes^\mathbf{L} \QQ/\ZZ)$ is
    of cofinite type and vanishes for $|i| \gg 0$, i.e. that
    $A^\bullet \otimes^\mathbf{L} \QQ/\ZZ$ is of cofinite type.

    Let us now consider the spectral sequences
    \begin{align}
      \label{eqn:homological-lemma-ss-1} E_2^{pq} & = H^p (G_\RR, H^q (A^\bullet \otimes \QQ)) \Longrightarrow H^{p+q} (G_\RR, A^\bullet \otimes \QQ),                     \\
      \label{eqn:homological-lemma-ss-2} E_2^{pq} & = \widehat{H}^p (G_\RR, H^q (A^\bullet \otimes \QQ)) \Longrightarrow \widehat{H}^{p+q} (G_\RR, A^\bullet \otimes \QQ).
    \end{align}
    We recall that $H^p (G_\RR, -)$ are $2$-torsion groups for $p > 0$. Since
    $H^q (A^\bullet \otimes \QQ)$ are $\QQ$-vector spaces, it follows that
    $E_2^{pq} = 0$ for $p > 0$ in \eqref{eqn:homological-lemma-ss-1}, and the
    spectral sequence degenerates. Similarly, the Tate cohomology groups
    $\widehat{H}^p (G_\RR, H^q (A^\bullet \otimes \QQ))$ are trivial for
    \emph{all} $p$ for the same reasons, so that
    \eqref{eqn:homological-lemma-ss-2} is trivial. This proves part 2).

    Part 3) now follows from the distinguished triangle
    \[ R\widehat{\Gamma} (G_\RR, A^\bullet) \to
      R\widehat{\Gamma} (G_\RR, A^\bullet \otimes \QQ) \to
      R\widehat{\Gamma} (G_\RR, A^\bullet \otimes^\mathbf{L} \QQ/\ZZ) \to
      R\widehat{\Gamma} (G_\RR, A^\bullet) [1]. \]

    Next, examining the spectral sequence $$E_2^{pq} = H^p (G_\RR, H^q (A^\bullet))
      \Longrightarrow H^{p+q} (G_\RR, A^\bullet),$$ we see that the groups $H^i
      (G_\RR, A^\bullet)$ are finitely generated, zero for $i \ll 0$, and torsion for
    $i \gg 0$. The latter is $2$-torsion. To see that, let $P_\bullet
      \twoheadrightarrow \ZZ$ be the bar-resolution of $\ZZ$ by free $\ZZ
      G_\RR$-modules. Consider the morphism of complexes
    \[ \begin{tikzcd}
        \cdots\ar{r} & P_3\ar{r}\ar{d}{2} & P_2\ar{r}\ar{d}{2} & P_1\ar{r}\ar{d}{2} & P_0\ar{r}\ar{d}{2-N} & 0 \\
        \cdots\ar{r} & P_3\ar{r} & P_2\ar{r} & P_1\ar{r} & P_0\ar{r} & 0
      \end{tikzcd} \]
    where $N$ denotes the norm map. The proof of \cite[Theorem~6.5.8]{Weibel-1994}
    shows that the above morphism induces multiplication by $2$ on $H^i (G_\RR,-)$
    for $i > 0$, and it is null-homotopic. Since $A^\bullet$ is bounded, we see
    that the above morphism induces multiplication by $2$ on $H^i (G_\RR,
      A^\bullet)$ for $i \gg 0$.

    Similarly, analyzing $$E_2^{pq} = H^p (G_\RR, H^q (A^\bullet \otimes^\mathbf{L}
      \QQ/\ZZ)) \Longrightarrow H^{p+q} (G_\RR, A^\bullet \otimes^\mathbf{L}
      \QQ/\ZZ).$$ we see that $H^i (G_\RR, A^\bullet \otimes^\mathbf{L} \QQ/\ZZ)$ are
    groups of cofinite type. To see that these are finite $2$-torsion for $i \gg
      0$, consider the triangle
    \[ R\Gamma (G_\RR, A^\bullet) \to
      R\Gamma (G_\RR, A^\bullet \otimes \QQ) \to
      R\Gamma (G_\RR, A^\bullet \otimes^\mathbf{L} \QQ/\ZZ) \to
      R\Gamma (G_\RR, A^\bullet) [1] \]
    Here $R\Gamma (G_\RR, A^\bullet \otimes \QQ)$ is bounded, and therefore $H^i
      (G_\RR, A^\bullet \otimes^\mathbf{L} \QQ/\ZZ) \cong H^{i+1} (G_\RR, A^\bullet)$
    for $i \gg 0$.
  \end{proof}
\end{lemma}

\begin{proposition}
  \label{prop:equivariant-coho-of-X(C)}
  Let $X$ be an arithmetic scheme and $n \in \ZZ$.
  Then $X (\CC)$ has the following types of complexes as its cohomology:

  \begin{center}
    \renewcommand{\arraystretch}{1.5}
    \begin{tabular}{|c|c|c|c|}
      \hline
                                                    & $A=\ZZ$                                                                                    & $A=\QQ$            & $A=\QQ/\ZZ$                                                                                  \\
      \hline
      $R\Gamma_c (X (\CC), A(n))$                   & perfect${}_{/\ZZ}$                                                                         & perfect${}_{/\QQ}$ & cofinite type                                                                                \\
      \hline
      $R\Gamma_c (G_\RR, X (\CC), A (n))$           & {\renewcommand{\arraystretch}{0.75}\begin{tabular}{c} almost \\ perfect \end{tabular}}     & perfect${}_{/\QQ}$ & {\renewcommand{\arraystretch}{0.75}\begin{tabular}{c} almost \\ cofinite type \end{tabular}} \\
      \hline
      $R\widehat{\Gamma}_c (G_\RR, X (\CC), A (n))$ & {\renewcommand{\arraystretch}{0.75}\begin{tabular}{c} finite \\ $2$-torsion \end{tabular}} & $\cong 0$          & {\renewcommand{\arraystretch}{0.75}\begin{tabular}{c} finite \\ $2$-torsion \end{tabular}}   \\
      \hline
    \end{tabular}
  \end{center}

  Moreover, there is an isomorphism
  \begin{equation}
    \label{eqn:Tate-vs-normal-cohomology-of-X(C)}
    \widehat{H}^i_c (G_\RR, X (\CC), \ZZ(n)) \cong
    H^i_c (G_\RR, X (\CC), \ZZ(n))
    \quad\text{for }i \ge 2 \dim X - 1.
  \end{equation}

  \begin{proof}
    We claim that $H^q_c (X (\CC), \ZZ(n))$ are finitely generated groups, and
    \begin{equation}
      \label{eqn:vanishing-of-cohomology-with-compact-support-of-X(C)}
      H^q_c (X (\CC), \ZZ(n)) = 0\quad\text{for }q \notin [0, 2 \dim X - 2].
    \end{equation}

    We may assume $X (\CC) \ne \emptyset$. The topological dimension of $X (\CC)$
    satisfies $\dim X = 1 + \dim X_\CC = 1 + \frac{1}{2} \dim_{top} X (\CC)$, so
    that $\dim_{top} X (\CC) = 2 \dim X - 2$.

    If $X (\CC)$ is smooth, we may assume it is of pure dimension $d = \dim_{top} X
      (\CC)$. Then finite generation and
    \eqref{eqn:vanishing-of-cohomology-with-compact-support-of-X(C)} follow from
    the Poincar\'{e} duality
    \[ H^i_c (X (\CC), \ZZ(n)) \cong H_{2d - i} (X (\CC), \ZZ(n)), \]
    and the fact that $X (\CC)$ has the homotopy type of a finite CW-complex by van
    der Waerden's theorem (see \cite{van-der-Waerden-30} and more recent
    expositions with more general results in \cite{Lojasiewicz-1964,
      Hironaka-1974}).

    In the general case, we use induction on the dimension of $X (\CC)$. Consider
    the decomposition $U (\CC) \hookrightarrow X (\CC) \hookleftarrow Z (\CC)$,
    where $Z (\CC)$ is the singular locus. In the long exact sequence
    \[ \cdots \to H^q_c (U (\CC), \ZZ(n)) \to H^q_c (X (\CC), \ZZ(n)) \to H^q_c (Z (\CC), \ZZ(n)) \to H^{q+1}_c (U (\CC), \ZZ(n)) \to \cdots \]
    the groups $H^q_c (U (\CC), \ZZ(n))$ are finitely generated by the smooth case,
    and $H^q_c (Z (\CC), \ZZ(n))$ are finitely generated by induction hypothesis.
    It follows that $H^q_c (X (\CC), \ZZ(n))$ are finitely generated. Similarly we
    conclude by induction that
    \eqref{eqn:vanishing-of-cohomology-with-compact-support-of-X(C)} holds.

    The rest of the table is an application of the previous lemma to $R\Gamma_c
      (X(\CC), \ZZ(n))$.

    Finally, \eqref{eqn:Tate-vs-normal-cohomology-of-X(C)} follows from the
    spectral sequences
    \begin{align*}
      \widehat{E}^{pq}_2 = \widehat{H}^p (G_\RR, H^q_c (X (\CC), \ZZ(n))) & \Longrightarrow
      \widehat{H}^i_c (G_\RR, X (\CC), \ZZ(n)),                                             \\
      E^{pq}_2 = H^p (G_\RR, H^q_c (X (\CC), \ZZ(n)))                     & \Longrightarrow
      H^i_c (G_\RR, X (\CC), \ZZ(n)),
    \end{align*}
    using \eqref{eqn:vanishing-of-cohomology-with-compact-support-of-X(C)}
    and the isomorphism
    $\widehat{H}^p (G_\RR, -) \cong H^p (G_\RR, -)$ for $p \ge 1$.
  \end{proof}
\end{proposition}


\section{Some consequences of Theorem~I}
\label{sec:consequences-of-theorem-I}

Now we deduce some consequences from the duality Theorem~\ref{theorem-I}.

\begin{lemma}
  \label{lemma:morphism-hat-Hc(Xet,Z(n))->Hc(Xet,Z(n))}
  The canonical morphism
  $\phi^i\colon \widehat{H}^i_c (X_\et, \ZZ (n)) \to H^i_c (X_\et, \ZZ (n))$
  sits in a long exact sequence
  \begin{multline*}
    \cdots \to \widehat{H}^{i-1}_c (G_\RR, X (\CC), \ZZ (n)) \to
    \widehat{H}_c^i (X_\et, \ZZ(n)) \xrightarrow{\phi^i}
    H_c^i (X_\et, \ZZ(n)) \\
    \to \widehat{H}^i_c (G_\RR, X (\CC), \ZZ (n)) \to \cdots
  \end{multline*}
  where the groups $\widehat{H}^i_c (G_\RR, X (\CC), \ZZ (n))$ are finite
  $2$-torsion. In particular,
  \begin{enumerate}
    \item[$1)$] the kernel and cokernel of $\phi^i$ are finite $2$-torsion,

    \item[$2)$] if $X (\RR) = \emptyset$, then
          $R\widehat{\Gamma}_c (G_\RR, X (\CC), \ZZ (n)) = 0$ and
          $\widehat{H}^i_c (X_\et, \ZZ (n)) \cong H^i_c (X_\et, \ZZ (n))$.
  \end{enumerate}

  \begin{proof}
    The exact sequence follows from the definition of modified \'{e}tale cohomology
    with compact support and Artin's comparison theorem. This is proved in
    \cite[Lemma~6.14]{Flach-Morin-2018}. In particular, the argument shows that
    $R\widehat{\Gamma}_c (G_\RR, X (\CC), \ZZ (n)) \cong
      R\widehat{\Gamma} (G_\RR, v^* Rf_* \ZZ(n))$ where
    $v\colon \Spec \CC \to \Spec \ZZ$ and $f\colon X\to \Spec \ZZ$,
    and $R\widehat{\Gamma}_c (G_\RR, X (\CC), \ZZ (n)) = 0$ if
    $X (\RR) = \emptyset$.

    The fact that $\widehat{H}^i_c (G_\RR, X (\CC), \ZZ (n))$ are finite
    $2$-torsion is a part of Proposition~\ref{prop:equivariant-coho-of-X(C)}.
  \end{proof}
\end{lemma}

\begin{proposition}
  \label{prop:motivic-cohomology-duality-consequences}
  Let $X$ be an arithmetic scheme of dimension $d$ satisfying Conjecture
  $\mathbf{L}^c (X_\et,n)$, let $n < 0$.

  \begin{enumerate}
    \item[$1)$] If $X (\RR) = \emptyset$, then $H^i (X_\et, \ZZ^c (n)) = 0$ for
          $i > 1$ or $i < -2d$.

    \item[$2)$] In general, $H^i (X_\et, \ZZ^c (n)) = 0$ for $i < -2d$, and
          $H^i (X_\et, \ZZ^c (n))$ is a finite $2$-torsion group for $i > 1$.

    \item[$3)$] If $X/\FF_q$ is a variety over a finite field, then the groups
          $H^i (X_\et, \ZZ^c(n))$ are finite for all $i \in \ZZ$.
  \end{enumerate}

  In general, we have the following cohomology for $n < 0$:
  \begin{center}
    \renewcommand{\arraystretch}{1.5}
    \begin{tabular}{|c|c|cl|cl|}
      \hline
      \textbf{groups}                    & \textbf{type}                                                                              & \multicolumn{2}{c|}{$i \ll 0$}                                                             & \multicolumn{2}{c|}{$i \gg 0$}                                                                                                                 \\
      \hline
      $H^i (X_\et, \ZZ^c (n))$           & {\renewcommand{\arraystretch}{0.75}\begin{tabular}{c} finitely \\ generated \end{tabular}} & $0$                                                                                        & for $i < -2d$                  & {\renewcommand{\arraystretch}{0.75}\begin{tabular}{c} finite \\ $2$-torsion \end{tabular}} & for $i > 1$      \\
      \hline
      $\widehat{H}^i_c (X_\et, \ZZ (n))$ & cofinite                                                                                   & {\renewcommand{\arraystretch}{0.75}\begin{tabular}{c} finite \\ $2$-torsion \end{tabular}} & for $i < 1$                    & $0$                                                                                        & for $i > 2d + 2$ \\
      \hline
      $H^i_c (X_\et, \ZZ (n))$           & cofinite                                                                                   & $0$                                                                                        & for $i < 1$                    & {\renewcommand{\arraystretch}{0.75}\begin{tabular}{c} finite \\ $2$-torsion \end{tabular}} & for $i > 2d + 2$ \\
      \hline
    \end{tabular}
  \end{center}
  In particular, $R\Gamma (X_\et, \ZZ^c (n))$ is an almost perfect complex,
  while $R\Gamma_c (X_\et, \ZZ (n))$ is almost of cofinite type in the sense of
  Definition~{\rm\ref{dfn:almost-of-(co)finite-type}}.

  \begin{proof}
    If $X (\RR) = \emptyset$, then our duality Theorem~\ref{theorem-I} gives
    \[ \Hom (H^{2-i} (X_\et, \ZZ^c (n)), \QQ/\ZZ) \cong
      \widehat{H}^i_c (X_\et, \ZZ (n)) \stackrel{X(\RR)=\emptyset}{\cong}
      H^i_c (X_\et, \ZZ (n)). \]
    We have $H^i_c (X_\et, \ZZ (n)) = 0$ for $i < 1$ by the definition of $\ZZ
      (n)$, and $H^i_c (X_\et, \ZZ (n)) = H^{i-1} (X_\et, \QQ/\ZZ(n)) = 0$ for $i >
      2d + 2$ for reasons of $\ell$-adic cohomological dimension \cite[Expos\'{e}~X,
    Th\'{e}or\`{e}me~6.2]{SGA4}. This proves part 1) of the proposition.

    In part 2), the group $H^i (X_\et, \ZZ^c (n))$ is finite $2$-torsion for $i >
      1$, thanks to part 1) and
    Lemma~\ref{lemma:morphism-hat-Hc(Xet,Z(n))->Hc(Xet,Z(n))}. Moreover, we have
    $H^i (X_\et, \ZZ^c (n)) \cong H^i (X_\et, \QQ^c (n))$ for $i < -2d$ according
    to \cite[Lemma~5.12]{Morin-2014}. Conjecture $L^c (X_\et, n)$ implies that
    these groups are $\QQ$-vector spaces finitely generated over $\ZZ$, hence
    trivial.

    In part 3), the cohomology groups $H^i (X_\et, \ZZ (n)) = H^{i-1} (X_\et,
      \QQ/\ZZ (n))$ are finite for $n < 0$ by \cite[Theorem~3]{Kahn-2003}.
  \end{proof}
\end{proposition}


\section{Complex $R\Gamma_\fg (X, \ZZ(n))$}
\label{sec:RGamma-fg}

The purpose of this section is to define auxiliary complexes $R\Gamma_\fg (X,
  \ZZ(n))$, which are used below in the construction of Weil-\'{e}tale
cohomology.

\begin{definition}
  \label{def:RGamma-fg}
  Assuming Conjecture $\mathbf{L}^c (X_\et,n)$ and $n < 0$,
  consider the morphism $\alpha_{X,n}$ in the derived category $\DZ$
  given by the composition
  \[ \begin{tikzcd}[column sep=4em]
      \RHom (R\Gamma (X_\et, \ZZ^c (n)), \QQ[-2]) \ar{r}{\QQ \twoheadrightarrow \QQ/\ZZ}\ar{ddr}[swap]{\alpha_{X,n}} & \RHom (R\Gamma (X_\et, \ZZ^c (n)), \QQ/\ZZ[-2]) \\
      & R\widehat{\Gamma}_c (X_\et, \ZZ (n)) \ar{u}{\text{Theorem~\ref{theorem-I}}}[swap]{\cong} \ar{d}{\text{proj.}} \\
      & R\Gamma_c (X_\et, \ZZ (n))
    \end{tikzcd} \]

  Here the first arrow is induced by the canonical projection $\QQ \to \QQ/\ZZ$,
  and the last arrow is the canonical projection from the modified cohomology
  with compact support to the usual cohomology with compact support (see
  Appendix~\ref{app:modified-cohomology-with-compact-support}).

  We define the complex $R\Gamma_\fg (X, \ZZ(n))$ as a cone of $\alpha_{X,n}$:
  \begin{multline*}
    \RHom (R\Gamma (X_\et, \ZZ^c (n)), \QQ [-2]) \xrightarrow{\alpha_{X,n}}
    R\Gamma_c (X_\et, \ZZ (n)) \to
    R\Gamma_\fg (X, \ZZ(n)) \\
    \to \RHom (R\Gamma (X_\et, \ZZ^c (n)), \QQ [-1])
  \end{multline*}
  Further, we denote
  $$H^i_\fg (X, \ZZ (n)) \dfn H^i (R\Gamma_\fg (X, \ZZ (n))).$$
\end{definition}

\begin{remark}
  \label{rmk:alpha-X-n-determined-by-cohomology}
  Under Conjecture $\mathbf{L}^c (X_\et, n)$, the groups
  $H^i_c (X_\et, \ZZ (n))$ for $n < 0$ are of cofinite type by Theorem~\ref{theorem-I},
  while $\RHom (R\Gamma (X_\et, \ZZ^c (n)), \QQ [-2])$ is a complex of
  $\QQ$-vector spaces. Therefore, the morphism $\alpha_{X,n}$ is completely
  determined by the maps between cohomology groups
  \[ H^i (\alpha_{X,n})\colon
    \Hom (H^{2-i} (X_\et, \ZZ^c (n)), \QQ) \to
    H^i_c (X_\et, \ZZ (n)) \]
  ---see Lemma~\ref{lemma:morphisms-in-DAb-between-cplx-of-Q-vs-and-almost-cofinite-type-cplx}.
\end{remark}

\begin{remark}
  We note that our $R\Gamma_\fg (X, \ZZ (n))$ plays the same role as
  $R\Gamma_W (\overline{X}_\et, \ZZ (n))$ in
  \cite[Definition~3.6]{Flach-Morin-2018}. We use a different notation since
  Flach and Morin work with the Artin--Verdier topology and their complex
  $R\Gamma_W (\overline{X}_\et, \ZZ (n))$ is perfect, while our complex can have
  finite $2$-torsion in arbitrarily high degree.
\end{remark}

We first note that the definition simplifies when $X$ has no real places.

\begin{proposition}
  \label{prop:RGamma-fg-for-X(R)-empty}
  If $X (\RR) = \emptyset$, then
  \[ R\Gamma_\fg (X, \ZZ (n)) \cong
    \RHom (R\Gamma (X_\et, \ZZ^c (n)), \ZZ [-1]). \]

  \begin{proof}
    In this case
    $R\widehat{\Gamma}_c (X_\et, \ZZ (n)) \to R\Gamma_c (X_\et, \ZZ (n))$
    is the identity morphism, and therefore $\alpha_{X,n}$ sits in the following
    commutative diagram with distinguished columns:
    \[ \begin{tikzcd}
        \RHom (R\Gamma (X_\et, \ZZ^c (n)), \QQ [-2])\ar{d}{\alpha_{X,n}} \ar{r}{\mathrm{id}} & \RHom (R\Gamma (X_\et, \ZZ^c (n)), \QQ [-2])\ar{d} \\
        R\Gamma_c (X_\et, \ZZ (n))\ar{d} \ar{r}{\cong}[swap]{\text{Theorem~\ref{theorem-I}}} & \RHom (R\Gamma (X_\et, \ZZ^c (n)), \QQ/\ZZ [-2])\ar{d} \\
        R\Gamma_\fg (X, \ZZ (n))\ar{d} \ar[dashed]{r}{\cong} & \RHom (R\Gamma (X_\et, \ZZ^c (n)), \ZZ [-1])\ar{d} \\
        \RHom (R\Gamma (X_\et, \ZZ^c (n)), \QQ [-1]) \ar{r}{\mathrm{id}} & \RHom (R\Gamma (X_\et, \ZZ^c (n)), \QQ [-1])
      \end{tikzcd} \]
    Here the first column is our definition of $R\Gamma_\fg (X, \ZZ (n))$, and the
    second column is induced by the distinguished triangle $\ZZ \to \QQ \to \QQ/\ZZ
      \to \ZZ [1]$.
  \end{proof}
\end{proposition}

\begin{proposition}
  \label{prop:RGammafg-almost-perfect}
  Assuming Conjecture $\mathbf{L}^c (X_\et, n)$, the complex
  $R\Gamma_\fg (X, \ZZ (n))$ for $n < 0$ is almost perfect in the sense of
  Definition~{\rm\ref{dfn:almost-of-(co)finite-type}}, i.e. its cohomology
  groups $H^i_\fg (X, \ZZ (n))$ are finitely generated, trivial for $i \ll 0$,
  and $2$-torsion for $i \gg 0$.

  \begin{proof}
    By the definition of $R\Gamma_\fg (X, \ZZ (n))$, there are short exact
    sequences
    \[ 0 \to \coker H^i (\alpha_{X,n}) \to
      H^i_\fg (X, \ZZ (n)) \to
      \ker H^{i+1} (\alpha_{X,n}) \to 0 \]

    The morphism $\alpha_{X,n}$ is given at the level of cohomology by
    \begin{multline*}
      H^i (\alpha_{X,n})\colon
      \Hom (H^{2-i} (X_\et, \ZZ^c (n)), \QQ) \xrightarrow{\psi^i}
      \Hom (H^{2-i} (X_\et, \ZZ^c (n)), \QQ/\ZZ) \xrightarrow{\cong} \\
      \widehat{H}^i_c (X_\et, \ZZ (n)) \xrightarrow{\phi^i} H^i_c (X_\et, \ZZ (n))
    \end{multline*}
    where $H^{2-i} (X_\et, \ZZ^c (n))$ is a finitely generated abelian group
    according to $\mathbf{L}^c (X_\et, n)$. We consider the ker-coker exact sequence
    (ignoring the isomorphism in the middle)
    \begin{multline*}
      0 \to \underbrace{\Hom (H^{2-i} (X_\et, \ZZ^c (n)), \ZZ)}_{\cong \ker \psi^i} \to
      \ker H^i (\alpha_{X,n}) \to
      \ker\phi^i \to \\
      \underbrace{\Hom (H^{2-i} (X_\et, \ZZ^c (n))_\tor, \QQ/\ZZ)}_{\cong \coker\psi^i} \to
      \coker H^i (\alpha_{X,n}) \to
      \coker\phi^i \to 0
    \end{multline*}

    Here $\ker\phi^i$ and $\coker\phi^i$ are finite $2$-torsion according to
    Lemma~\ref{lemma:morphism-hat-Hc(Xet,Z(n))->Hc(Xet,Z(n))}, and $H^\bullet
      (X_\et, \ZZ^c (n))$ are finitely generated by $\mathbf{L}^c (X_\et, n)$. This
    establishes finite generation of $\ker H^{i+1} (\alpha_{X,n})$ and $\coker H^i
      (\alpha_{X,n})$, and hence of $H^i_\fg (X, \ZZ (n))$.

    From the description of cohomology groups in
    Proposition~\ref{prop:motivic-cohomology-duality-consequences}, for $i \ll 0$
    we have $\Hom (H^{2-i} (X_\et, \ZZ^c (n)), \QQ) = H^i_c (X_\et, \ZZ(n)) = 0$,
    and hence $H^i_\fg (X, \ZZ (n)) = 0$. On the other hand, for $i \gg 0$ we have
    $\Hom (H^{2-i} (X_\et, \ZZ^c (n)), \QQ) = 0$, so that $H^i_\fg (X, \ZZ (n))
      \cong H^i_c (X_\et, \ZZ (n))$ is a finite $2$-torsion group.
  \end{proof}
\end{proposition}

\begin{proposition}
  \label{prop:RGamma-fg-well-defined}
  The complex $R\Gamma_\fg (X, \ZZ (n))$ is defined up to a unique isomorphism
  in the derived category $\DZ$.

  \begin{proof}
    The complex $\RHom (R\Gamma (X_\et, \ZZ^c (n)), \QQ [-2])$ consists of
    $\QQ$-vector spaces, and $R\Gamma_\fg (X, \ZZ (n))$ is almost perfect, so we
    are in the situation of Corollary~\ref{cor:TR3-TR1-with-uniqueness}.
  \end{proof}
\end{proposition}

\begin{proposition}
  \label{prop:tensoring-RGammafg-with-Z/m-and-Q}
  Suppose that Conjecture $\mathbf{L}^c (X_\et,n)$ holds and consider the
  distinguished triangle defining $R\Gamma_\fg (X, \ZZ (n))$ for $n < 0$:
  \begin{multline*}
    \RHom (R\Gamma (X_\et, \ZZ^c (n)), \QQ [-2]) \xrightarrow{\alpha_{X,n}}
    R\Gamma_c (X_\et, \ZZ (n)) \xrightarrow{f}
    R\Gamma_\fg (X, \ZZ (n)) \\
    \xrightarrow{g} \RHom (R\Gamma (X_\et, \ZZ^c (n)), \QQ [-1])
  \end{multline*}

  \begin{enumerate}
    \item[$1)$] The morphism $g$ induces an isomorphism
          \[ g\otimes \QQ\colon R\Gamma_\fg (X, \ZZ (n)) \otimes \QQ \xrightarrow{\cong}
            \RHom (R\Gamma (X_\et, \ZZ^c (n)), \QQ [-1]).\]

    \item[$2)$] For each $m \ge 1$ the morphism $f$ induces an
          isomorphism
          \[ f\otimes \ZZ/m\ZZ\colon
            R\Gamma_c (X_\et, \ZZ (n))\otimes^\mathbf{L} \ZZ/m\ZZ \xrightarrow{\cong}
            R\Gamma_\fg (X, \ZZ (n))\otimes^\mathbf{L} \ZZ/m\ZZ \]

    \item[$3)$] For any prime $\ell$ the morphism $f$ induces an isomorphism
          $$\varprojlim_r H_c^i (X_\et, \ZZ/\ell^r (n)) \cong H_\fg^i (X, \ZZ (n)) \otimes \ZZ_\ell.$$
  \end{enumerate}

  \begin{proof}
    The groups $H_c^i (X_\et, \ZZ (n))$ are all torsion, and therefore
    $R\Gamma_c (X_\et, \ZZ (n)) \otimes \QQ \cong 0$ in the derived
    category. Similarly, the complexes of $\QQ$-vector spaces
    $\RHom (R\Gamma (X_\et, \ZZ^c (n)), \QQ [\cdots])$ are killed by tensoring
    with $\ZZ/m\ZZ$.  This proves 1) and 2).

    Now 2) implies 3): by the finite generation of $H_\fg^i (X, \ZZ (n))$, we have
    \[ \varprojlim_r H_c^i (X_\et, \ZZ/\ell^r (n)) \stackrel{\text{2)}}{\cong}
      \varprojlim_r H_\fg^i (X, \ZZ/\ell^r (n)) \cong
      \varprojlim_r H_\fg^i (X, \ZZ (n))/\ell^r \cong
      H_\fg^i (X, \ZZ (n)) \otimes \ZZ_\ell. \qedhere \]
  \end{proof}
\end{proposition}

The groups $H_\fg^i (X, \ZZ (n))$ provide an integral model for $\ell$-adic
cohomology in the following sense (see also \cite[\S 8]{Geisser-2004}).

\begin{corollary}
  \label{cor:RGamma-fg-model-for-l-adic-cohomology}
  Let $X$ be an arithmetic scheme satisfying Conjecture
  $\mathbf{L}^c (X_\et, n)$ for $n < 0$. Then
  $$H_\fg^i (X, \ZZ (n)) \otimes \ZZ_\ell \cong H^i_c (X [1/\ell]_\et, \ZZ_\ell (n)),$$
  where the right-hand side denotes $\ell$-adic cohomology with compact support.

  \begin{proof}
    We have $\ZZ (n)/\ell^r \cong j_{\ell!} \mu_m^{\otimes n}$.
    Now by part 3) of the previous proposition,
    \[ H_\fg^i (X, \ZZ (n)) \otimes \ZZ_\ell \cong
      \varprojlim_r H_c^i (X_\et, j_{\ell!} \mu_{\ell^r}^{\otimes n}) \cong
      \varprojlim_r H_c^i (X [1/\ell]_\et, \mu_{\ell^r}^{\otimes n})
      \stackrel{\text{dfn}}{=} H_c^i (X [1/\ell]_\et, \ZZ_\ell (n)). \qedhere \]
  \end{proof}
\end{corollary}


\section{Proof of Theorem~II}
\label{sec:theorem-II}

The aim of this section is to prove Theorem~\ref{theorem-II}. We recall that it
states that the morphism of complexes $u_\infty^*$, defined as the composition
\[ \begin{tikzcd}
    R\Gamma_c (X_\et, \ZZ(n)) \ar[equals]{d}\ar[dashed]{r}{u_\infty^*} & R\Gamma_c (G_\RR, X (\CC), \ZZ (n)) \\
    R\Gamma_c (X_\et, \QQ/\ZZ (n)) [-1] \ar{r}{v_\infty^* [-1]} & R\Gamma_c (G_\RR, X (\CC), \QQ/\ZZ (n)) [-1] \ar{u}
  \end{tikzcd} \]
is torsion. Here $v_\infty^*\colon R\Gamma_c (X_\et, \QQ/\ZZ (n)) \to R\Gamma_c
  (G_\RR, X (\CC), \QQ/\ZZ (n))$ is induced by the comparison functor
$\alpha^*\colon \mathbf{Sh} (X_\et) \to \mathbf{Sh} (G_\RR, X (\CC))$, as
explained in Proposition~\ref{prop:inverse-image-gamma}. We first ensure that
$\alpha^*$ identifies the sheaf $\QQ/\ZZ (n)$ on $X_\et$ from
Definition~\ref{dfn:sheaf-Z(n)} with the $G_\RR$-equivariant sheaf $\QQ/\ZZ (n)
  \dfn \frac{(2\pi i)^n\,\QQ}{(2\pi i)^n\,\ZZ}$ on $X (\CC)$.

\begin{proposition}
  \label{propn:image-of-Q/Zn-under-alpha}
  For the sheaf $\QQ/\ZZ (n)$ on $X_\et$ we have an isomorphism of
  $G_\RR$-equivariant constant sheaves on $X (\CC)$
  $$\alpha^* \QQ/\ZZ (n) \cong \QQ/\ZZ (n).$$

  \begin{proof}
    We first compute that the functor $\alpha^*$ sends the sheaf
    $\mu_m^{\otimes n}$ on $X_\et$ to the constant $G_\RR$-equivariant sheaf
    $\frac{(2\pi i)^n\,\ZZ}{m\,(2\pi i)^n\,\ZZ}$ on $X(\CC)$:
    \begin{align*}
      \alpha^* \mu_m^{\otimes n} & \cong \mu_m (\CC)^{\otimes n} \dfn \iHom (\mu_m (\CC)^{\otimes (-n)}, \ZZ/m\ZZ) \\
                                 & \cong \frac{(2\pi i)^n\,\ZZ}{m\,(2\pi i)^n\,\ZZ}
    \end{align*}
    ---here the first isomorphism comes from the definition of $\alpha^*$ given
    in Appendix~\ref{app:modified-cohomology-with-compact-support}, and the
    second isomorphism comes from the corresponding isomorphism of
    $G_\RR$-modules.

    Since $\alpha^*$ preserves colimits
    (Lemma~\ref{lemma:alpha-preserves-colimits}), we have
    \[
      \alpha^* \QQ/\ZZ (n)
      = \alpha^* \Bigl(\bigoplus_p \varinjlim_r j_{p!} \mu_{p^r}^{\otimes n}\Bigr)
      \cong \varinjlim_m \alpha^* \mu_m^{\otimes n}
      \cong \varinjlim_m \frac{(2\pi i)^n\,\ZZ}{m\,(2\pi i)^n\,\ZZ}
      \cong \frac{(2\pi i)^n\,\QQ}{(2\pi i)^n\,\ZZ}.
      \qedhere
    \]
  \end{proof}
\end{proposition}

We proceed with our proof of Theorem~\ref{theorem-II}. Our argument follows the
proof of \cite[Lemma 3.25]{Flach-Morin-2018}. We'll need following result about
$\ell$-adic cohomology.

\begin{proposition}
  \label{prop:l-adic-cohomology-key-lemma}
  Let $X$ be an arithmetic scheme and $n < 0$. Then for any prime $\ell$ we have
  $$(H^i_c (X_{\overline{\QQ},\text{\it \'{e}t}}, \QQ_\ell/\ZZ_\ell (n))^{G_\QQ})_\div = 0.$$

  \begin{proof}
    We claim that for a suitable choice of prime $p \ne \ell$,
    \[
      H^i_c (X_{\overline{\QQ}, \text{\it \'{e}t}}, \ZZ_\ell (n)) \cong
      H^i_c (X_{\overline{\FF_p}, \text{\it \'{e}t}}, \ZZ_\ell (n))
      \quad\text{for all }i.
    \]

    We have $f\colon X \to \Spec \ZZ$, separated, of finite type. $\ZZ_\ell (n)$ is
    a constructible $\ZZ_\ell$-sheaf on $X$ in the sense of \cite[Expos\'{e} VI,
    1.1.1]{SGA5}, and by [ibid., 2.2.2], $R^i f_! \ZZ_\ell (n)$ is a constructible
    $\ZZ_\ell$-sheaf on $\Spec \ZZ$. Now [ibid., 1.2.6] implies that there exists
    an open subscheme $U = \Spec \ZZ_S \subset \Spec \ZZ$ such that $R^i f_!
      \ZZ_\ell (n)$ is a twisted constant sheaf on $U$. We may take a finite set of
    primes $S$ such that this holds for all $i$. Then for $p \notin S$, the proper
    base change for constructible sheaves [ibid. 2.2.3] applied to the diagram
    \[ \begin{tikzcd}
        X_{\overline{\QQ}} \ar{r}\tikzpb\ar{d} & X_U \ar{d}{f_U} & X_{\overline{\FF_p}} \ar{l}\ar{d}\tikzpbur \\
        \Spec \overline{\QQ} \ar{r}[swap]{\overline{\eta}} & \Spec \ZZ_S & \Spec \overline{\FF_p} \ar{l}{\overline{x}}
      \end{tikzcd} \]
    gives us an isomorphism
    \begin{equation}
      \label{eqn:iso-pbc-Zl-Gal(QS/Q)}
      H^i_c (X_{\overline{\QQ}, \text{\it \'{e}t}}, \ZZ_\ell (n)) \cong
      (R^i f_{U,!} \ZZ_\ell (n))_{\overline{\eta}} \cong
      (R^i f_{U,!} \ZZ_\ell (n))_{\overline{x}} \cong
      H^i_c (X_{\overline{\FF_p}, \text{\it \'{e}t}}, \ZZ_\ell (n)).
    \end{equation}

    We denote by $I_p$ the inertia subgroup of the absolute Galois group
    $G_{\QQ_p}$: $$1 \to I_p \to G_{\QQ_p} \to G_{\FF_p} \to 1$$

    The isomorphism \eqref{eqn:iso-pbc-Zl-Gal(QS/Q)} is equivariant under the
    $G_{\QQ_p}$-action on the left-hand side and $G_{\FF_p}$-action on the
    right-hand side. We have
    \[ H^i_c (X_{\overline{\QQ}, \text{\it \'{e}t}}, \QQ_\ell/\ZZ_\ell (n))^{G_\QQ} \rightarrowtail
      H^i_c (X_{\overline{\QQ}, \text{\it \'{e}t}}, \QQ_\ell/\ZZ_\ell (n))^{G_{\QQ_p}/I_p}
      \cong H^i_c (X_{\overline{\FF_p}, \text{\it \'{e}t}}, \QQ_\ell/\ZZ_\ell (n))^{G_{\FF_p}}, \]
    so it suffices to show that $$(H^i_c (X_{\overline{\FF_p},\text{\it \'{e}t}},
      \QQ_\ell/\ZZ_\ell (n))^{G_{\FF_p}})_\div = 0.$$

    The long exact sequence of $G_{\FF_p}$-modules
    \begin{multline*}
      \cdots \to
      H^i_c (X_{\overline{\FF_p},\text{\it \'{e}t}}, \ZZ_\ell (n)) \to
      H^i_c (X_{\overline{\FF_p},\text{\it \'{e}t}}, \QQ_\ell (n)) \to
      H^i_c (X_{\overline{\FF_p},\text{\it \'{e}t}}, \QQ_\ell/\ZZ_\ell (n)) \\
      \to H^{i+1}_c (X_{\overline{\FF_p},\text{\it \'{e}t}}, \ZZ_\ell (n)) \to
      \cdots
    \end{multline*}
    induces short exact sequences
    \begin{equation}
      \label{eqn:Zl-Ql-Ql/Zl-ses}
      0 \to H^i_c (X_{\overline{\FF_p},\text{\it \'{e}t}}, \ZZ_\ell (n))_\cotor \to
      H^i_c (X_{\overline{\FF_p},\text{\it \'{e}t}}, \QQ_\ell (n)) \to
      H^i_c (X_{\overline{\FF_p},\text{\it \'{e}t}}, \QQ_\ell/\ZZ_\ell (n))_\div \to 0
    \end{equation}
    Here $H^i_c (X_{\overline{\FF_p},\text{\it \'{e}t}}, \ZZ_\ell (n))_\cotor \dfn
      \frac{H^i_c (X_{\overline{\FF_p},\text{\it \'{e}t}}, \ZZ_\ell (n))}{H^i_c (X_{\overline{\FF_p},\text{\it \'{e}t}}, \ZZ_\ell (n))_\tor}$,
    and we use that $H^i_c (X_{\overline{\FF_p},\text{\it \'{e}t}}, \ZZ_\ell (n))$
    are finitely generated $\ZZ_\ell$-modules, hence have no nontrivial divisible subgroups.

    According to \cite[Expos\'{e}~XXI, 5.5.3]{SGA7}, the eigenvalues of the
    geometric Frobenius acting on $H^i_c (X_{\overline{\FF_p},\text{\it \'{e}t}},
      \QQ_\ell)$ are algebraic integers. After twisting $\QQ_\ell$ by $n$, the
    eigenvalues will lie in $p^{-n}\,\overline{\ZZ}$. Since $n < 0$ by our
    assumption, this implies that $1$ does not appear as an eigenvalue, and hence
    $$H^i_c (X_{\overline{\FF_p},\text{\it \'{e}t}}, \QQ_\ell (n))^{G_{\FF_p}} =
      0.$$ Thus, after taking the $G_{\FF_p}$-invariants in
    \eqref{eqn:Zl-Ql-Ql/Zl-ses}, we obtain
    \[
      0 \to (H^i_c (X_{\overline{\FF_p},\text{\it \'{e}t}}, \QQ_\ell/\ZZ_\ell (n))_\div)^{G_{\FF_p}} \to
      H^1 (G_{\FF_p}, H^i_c (X_{\overline{\FF_p},\text{\it \'{e}t}}, \ZZ_\ell (n))_\cotor) \to \cdots
    \]
    This gives a monomorphism between the maximal divisible subgroups
    \[ ((H^i_c (X_{\overline{\FF_p},\text{\it \'{e}t}}, \QQ_\ell/\ZZ_\ell (n))_\div)^{G_{\FF_p}})_\div \rightarrowtail
      H^1 (G_{\FF_p}, H^i_c (X_{\overline{\FF_p},\text{\it \'{e}t}}, \ZZ_\ell (n))_\cotor)_\div. \]
    However, $H^1 (G_{\FF_p}, H^i_c (X_{\overline{\FF_p},\text{\it \'{e}t}},
      \ZZ_\ell (n))_\cotor)$ is a finitely generated $\ZZ_\ell$-module, and therefore
    its maximal divisible subgroup is trivial. We conclude that
    \[ (H^i_c (X_{\overline{\FF_p},\text{\it \'{e}t}}, \QQ_\ell/\ZZ_\ell (n))^{G_{\FF_p}})_\div =
      ((H^i_c (X_{\overline{\FF_p},\text{\it \'{e}t}}, \QQ_\ell/\ZZ_\ell (n))_\div)^{G_{\FF_p}})_\div = 0. \qedhere \]
  \end{proof}
\end{proposition}

\begin{proof}[Proof of Theorem~\ref{theorem-II}]
  By Definition~\ref{dfn:u-infty}, this amounts to showing that the morphism
  $$v_\infty^*\colon R\Gamma_c (X_\et, \QQ/\ZZ (n)) \to R\Gamma_c (G_\RR, X (\CC), \QQ/\ZZ (n))$$
  is torsion. The complexes $R\Gamma_c (X_\et, \QQ/\ZZ (n))$ and
  $R\Gamma_c (G_\RR, X (\CC), \QQ/\ZZ (n))$ are almost of cofinite type by
  Proposition~\ref{prop:motivic-cohomology-duality-consequences} and
  Proposition~\ref{prop:equivariant-coho-of-X(C)} respectively.
  Therefore, according to Lemma~\ref{lemma:torsion-morphisms-in-D(Z)}, to show
  that
  $v^*_\infty\colon R\Gamma_c (X_\et, \QQ/\ZZ (n)) \to R\Gamma_c (G_\RR, X (\CC), \QQ/\ZZ (n))$
  is torsion, it suffices to show that the corresponding morphisms on the
  maximal divisible subgroups
  \[ H^i_c (v^*_\infty)_\div\colon H^i_c (X_\et, \QQ/\ZZ (n))_\div \to
    H^i_c (G_\RR, X (\CC), \QQ/\ZZ (n))_\div \]
  are trivial. The morphism $H^i_c (X_\et, \QQ/\ZZ (n)) \to H^i_c
    (X_{\overline{\QQ}, \text{\it \'{e}t}}, \QQ/\ZZ (n))$, and hence $H^i_c
    (v^*_\infty)$, factors through $H^i_c (X_{\overline{\QQ}, \text{\it \'{e}t}},
    \mu^{\otimes n})^{G_\QQ}$, where $\mu^{\otimes n}$ is the sheaf of all roots of
  unity on $X_{\overline{\QQ}, \text{\it \'{e}t}}$ twisted by $n$. So we have
  \[ \begin{tikzcd}[column sep=0pt]
      H^i_c (X_\et, \QQ/\ZZ (n))_\div\ar{rr}{H^i_c (v^*_\infty)_\div}\ar[dashed]{dr} && H^i_c (G_\RR, X (\CC), \QQ/\ZZ (n))_\div\\
      & \left(H^i_c (X_{\overline{\QQ}, \text{\it \'{e}t}}, \mu^{\otimes n})^{G_\QQ}\right)_\div\ar[dashed]{ur}
    \end{tikzcd} \]
  Now
  \begin{align*}
    \left(H^i_c (X_{\overline{\QQ}, \text{\it \'{e}t}}, \mu^{\otimes n})^{G_\QQ}\right)_\div & \cong
    \left(\bigoplus_\ell H^i_c (X_{\overline{\QQ}, \text{\it \'{e}t}}, \QQ_\ell/\ZZ_\ell (n))^{G_\QQ}\right)_\div \\
                                                                                             & \cong
    \bigoplus_\ell \left(H^i_c (X_{\overline{\QQ}, \text{\it \'{e}t}}, \QQ_\ell/\ZZ_\ell (n))^{G_\QQ}\right)_\div,
  \end{align*}
  where all the summands are trivial by
  Proposition~\ref{prop:l-adic-cohomology-key-lemma}.
\end{proof}


\section{Weil-\'{e}tale complex $R\Gamma_\Wc (X, \ZZ(n))$} \label{sec:RGamma-Wc}

The aim of this section is to construct the Weil-\'{e}tale cohomology complexes
$R\Gamma_\Wc (X, \ZZ(n))$.

\begin{lemma}
  Let $X$ be an arithmetic scheme and $n < 0$. Assume Conjecture
  $\mathbf{L}^c (X_\et, n)$, so that the morphism $\alpha_{X,n}$ exists.
  Then $u_\infty^* \circ \alpha_{X,n} = 0$.

  \[ \begin{tikzcd}
      \RHom (R\Gamma (X, \ZZ^c (n)), \QQ [-2]) \ar{d}[swap]{\alpha_{X,n}}\ar{dr}{= 0} \\
      R\Gamma_c (X_\et, \ZZ (n)) \ar{r}{u_\infty^*} & R\Gamma_c (G_\RR, X (\CC), \ZZ (n))
    \end{tikzcd} \]

  \begin{proof}
    The morphism $\alpha_{X,n}$ is defined on a complex of $\QQ$-vector spaces,
    and $u_\infty^*$ is torsion by Theorem~\ref{theorem-II}.
  \end{proof}
\end{lemma}

\begin{definition}
  \label{dfn:i-infty}
  We let
  $i_\infty^*\colon R\Gamma_\fg (X, \ZZ (n)) \to R\Gamma_c (G_\RR, X (\CC), \ZZ (n))$
  be a morphism in $\DZ$ that gives a morphism of distinguished triangles
  \begin{equation}
    \label{eqn:triangle-defining-i-infty}
    \begin{tikzcd}
      \RHom (R\Gamma (X, \ZZ^c (n)), \QQ [-2]) \ar{d}[swap]{\alpha_{X,n}}\ar{r} & 0\ar{d} \\
      R\Gamma_c (X_\et, \ZZ (n)) \ar{r}{u_\infty^*}\ar{d} &  R\Gamma_c (G_\RR, X (\CC), \ZZ (n)) \ar{d}{id} \\
      R\Gamma_\fg (X, \ZZ (n)) \ar[dashed]{r}{i_\infty^*}\ar{d} & R\Gamma_c (G_\RR, X (\CC), \ZZ (n)) \ar{d} \\
      \RHom (R\Gamma (X, \ZZ^c (n)), \QQ [-1])\ar{r} & 0 \\
    \end{tikzcd}
  \end{equation}
\end{definition}

In fact, this makes $i_\infty^*$ independent of any choices.

\begin{proposition}
  \label{prop:uniqueness-of-i-infty}
  There is a unique morphism $i_\infty^*$ that fits in the
  diagram \eqref{eqn:triangle-defining-i-infty}.

  \begin{proof}
    We can apply Corollary~\ref{cor:TR3-TR1-with-uniqueness}, since
    $\RHom (R\Gamma (X, \ZZ^c (n)), \QQ [-2])$ is a complex of $\QQ$-vector
    spaces, and both
    $R\Gamma_\fg (X, \ZZ (n))$ and
    $R\Gamma_c (G_\RR, X (\CC), \ZZ (n))$
    are almost perfect by Proposition~\ref{prop:RGammafg-almost-perfect} and
    Proposition~\ref{prop:equivariant-coho-of-X(C)}.
  \end{proof}
\end{proposition}

\begin{proposition}
  \label{i-infty-is-torsion}
  The morphism $i_\infty^*$ is torsion.

  \begin{proof}
    Let us examine the morphism of distinguished triangles
    \eqref{eqn:triangle-defining-i-infty} that defines $i_\infty^*$; in
    particular, the commutative diagram
    \[ \begin{tikzcd}
        R\Gamma_c (X_\et, \ZZ (n)) \ar{r}\ar{d}[swap]{u_\infty^*} & R\Gamma_\fg (X, \ZZ (n))\ar{dl}{i_\infty^*} \\
        R\Gamma_c (G_\RR, X (\CC), \ZZ (n))
      \end{tikzcd} \]

    According to Corollary~\ref{cor:TR3-TR1-with-uniqueness}, the morphism
    \begin{multline*}
      \Hom_\DZ (R\Gamma_\fg (X,\ZZ (n)), R\Gamma_c (G_\RR, X (\CC), \ZZ (n))) \\
      \to
      \Hom_\DZ (R\Gamma_c (X_\et, \ZZ (n)), R\Gamma_c (G_\RR, X (\CC), \ZZ (n)))
    \end{multline*}
    induced by the composition with
    $R\Gamma_c (X_\et, \ZZ (n)) \to R\Gamma_\fg (X,\ZZ (n))$, is a monomorphism, and
    therefore
    \begin{multline*}
      \Hom_\DZ (R\Gamma_\fg (X,\ZZ (n)), R\Gamma_c (G_\RR, X (\CC), \ZZ (n)))\otimes \QQ \to\\
      \Hom_\DZ (R\Gamma_c (X_\et, \ZZ (n)), R\Gamma_c (G_\RR, X (\CC), \ZZ (n)))\otimes \QQ
    \end{multline*}
    is also a monomorphism. However, $u_\infty^*\otimes \QQ = 0$ by
    Theorem~\ref{theorem-II}, and this implies that $i_\infty^*\otimes \QQ = 0$.
  \end{proof}
\end{proposition}

We are now ready to define the Weil-\'{e}tale complexes.

\begin{definition}
  \label{dfn:RGammaWc}
  We let
  $R\Gamma_\Wc (X,\ZZ(n))$ be an object in the derived category
  $\DZ$ which is a mapping fiber of $i_\infty^*$:
  \[ R\Gamma_\Wc (X,\ZZ(n)) \to
    R\Gamma_\fg (X, \ZZ (n)) \xrightarrow{i_\infty^*}
    R\Gamma_c (G_\RR, X (\CC), \ZZ (n)) \to
    R\Gamma_\Wc (X,\ZZ(n)) [1]. \]
  The \textbf{Weil-\'{e}tale cohomology with compact support} is given by
  $$H_\Wc^i (X, \ZZ (n)) \dfn H^i (R\Gamma_\Wc (X,\ZZ(n))).$$
\end{definition}

\begin{remark}
  Note that this defines $R\Gamma_\Wc (X,\ZZ(n))$ up to a non-unique isomorphism
  in $\DZ$, and the groups $H_\Wc^i (X, \ZZ (n))$ are also defined
  up to a non-unique isomorphism. In a continuation of this paper we will make
  use of the determinant $\det\nolimits_\ZZ R\Gamma_\Wc (X,\ZZ(n))$ in the
  sense of \cite{Knudsen-Mumford-1976}, which will be defined up to a canonical
  isomorphism.

  However, we recall from Proposition~\ref{prop:RGamma-fg-well-defined} that
  $R\Gamma_\fg (X, \ZZ (n))$ is defined up to a unique isomorphism in the derived
  category $\DZ$. If we could define $i_\infty^*\colon R\Gamma_\fg (X, \ZZ(n))
    \to R\Gamma_c (G_\RR, X(\CC), \ZZ(n))$ as an explicit, genuine morphism of
  complexes (not just as a morphism in the derived category $\DZ$), this would
  give us a canonical and functorial definition for $R\Gamma_\Wc (X, \ZZ(n))$.
\end{remark}

\subsection*{Case of varieties over finite fields}

For varieties over finite fields, our Weil-\'{e}tale cohomology has a simple
description, and it is $\QQ/\ZZ$-dual to the arithmetic homology studied by
Geisser in \cite{Geisser-2010-arithmetic-homology}.

\begin{proposition}
  If $X$ is a variety over a finite field $\FF_q$ and $n < 0$, then assuming
  Conjecture $\mathbf{L}^c (X,n)$, there is an isomorphism of complexes
  \begin{equation}
    \label{eqn:RGamma-Wc-over-finite-fields}
    R\Gamma_\Wc (X,\ZZ(n)) \cong \RHom (R\Gamma (X_\et, \ZZ^c (n)), \ZZ [-1]),
  \end{equation}
  and an isomorphism of finite groups
  \begin{align*}
    H^i_{W,c} (X, \ZZ (n)) & \cong
    \Hom (H^{2-i} (X_\et, \ZZ^c (n)), \QQ/\ZZ) \\
                           & \cong
    H^i_c (X_\et, \ZZ(n))                      \\
                           & \cong
    \Hom (H_{i-1}^c (X_\ar, \ZZ (n)), \QQ/\ZZ),
  \end{align*}
  where $H_\bullet^c (X_\ar, \ZZ (n))$ are the arithmetic homology groups
  defined in {\rm \cite[\S 3]{Geisser-2010-arithmetic-homology}}.

  \begin{proof}
    Under our assumptions, $X (\CC) = \emptyset$, and therefore
    $R\Gamma_c (G_\RR, X (\CC), \ZZ (n)) = 0$, so that
    $R\Gamma_\Wc (X,\ZZ(n)) \cong R\Gamma_\fg (X, \ZZ (n))$. Finally, by
    Proposition~\ref{prop:RGamma-fg-for-X(R)-empty}, we have an isomorphism
    $R\Gamma_\fg (X, \ZZ (n)) \cong \RHom (R\Gamma (X_\et, \ZZ^c (n)), \ZZ [-1])$.

    To relate this to Geisser's arithmetic homology, according to
    \cite[Theorem~3.1]{Geisser-2010-arithmetic-homology}, there is a long exact
    sequence
    \[ \cdots \to H_{i-1}^c (X_\et, \ZZ (n)) \to
      H_i^c (X_\ar, \ZZ (n)) \to CH_n (X, i-2n)_\QQ \to
      H_{i-2}^c (X_\et, \ZZ (n)) \to \cdots \]
    Here the homological notation means that
    \begin{align*}
      H_i^c (X_\et, \ZZ (n)) & = H^{-i} (X_\et, \ZZ^c (n)),  \\
      CH_n (X, i-2n)_\QQ     & = H_i^c (X_\et, \QQ (n)) = 0,
    \end{align*}
    where the rational vanishing uses finiteness of $H^i (X_\et, \ZZ^c(n))$
    for $X$ over a finite field and $n < 0$,
    assuming $\mathbf{L}^c (X_\et,n)$
    (Proposition~\ref{prop:motivic-cohomology-duality-consequences}).

    Therefore, $$H_i^c (X_\ar, \ZZ (n)) \cong H^{1-i} (X_\et, \ZZ^c (n)).$$

    Now \eqref{eqn:RGamma-Wc-over-finite-fields} gives
    \[
      E_2^{p,q} = \Ext_\ZZ^p (H^{1-q} (X_\et, \ZZ^c (n)), \ZZ)
      \Longrightarrow H^{p+q}_{W,c} (X, \ZZ (n)),
    \]
    and again, by finiteness of $H^{1-q} (X_\et, \ZZ^c (n))$ under our assumptions,
    this spectral sequence is concentrated in $p = 1$, where
    \[ \Ext_\ZZ^1 (H^{1-q} (X_\et, \ZZ^c (n)), \ZZ) \cong
      \Hom (H^{1-q} (X_\et, \ZZ^c (n)), \QQ/\ZZ), \]
    so that
    \[ H^{1+i}_{W,c} (X, \ZZ (n)) \cong
      \Hom (H^{1-i} (X_\et, \ZZ^c (n)), \QQ/\ZZ) \cong
      \Hom (H_i^c (X_\ar, \ZZ (n)), \QQ/\ZZ). \qedhere \]
  \end{proof}
\end{proposition}

\subsection*{Perfectness of the complex}

Our next aim is to verify that $R\Gamma_\Wc (X, \ZZ(n))$ is a perfect complex.
From now on we tacitly assume Conjecture $\mathbf{L}^c (X_\et,n)$. As always,
$n < 0$.

\begin{lemma}
  The groups $H^i_\Wc (X, \ZZ(n))$ are finitely generated for all $i \in \ZZ$.

  \begin{proof}
    In the long exact sequence
    \begin{multline*}
      \cdots \to H^{i-1}_c (G_\RR, X (\CC), \ZZ (n)) \to
      H^i_\Wc (X,\ZZ(n)) \to
      H^i_\fg (X,\ZZ(n)) \\
      \xrightarrow{H^i (i_\infty^*)}
      H^i_c (G_\RR, X (\CC), \ZZ (n)) \to \cdots
    \end{multline*}
    the groups $H^i_c (G_\RR, X (\CC), \ZZ (n))$ and $H^i_\fg (X, \ZZ(n))$ are
    finitely generated by
    Proposition~\ref{prop:equivariant-coho-of-X(C)}, and
    Proposition~\ref{prop:RGammafg-almost-perfect}, respectively.
    This implies the finite generation of $H^i_\Wc (X, \ZZ(n))$.
  \end{proof}
\end{lemma}

\begin{lemma}
  One has $H^i_\Wc (X,\ZZ(n)) = 0$ for $i < 0$.

  \begin{proof}
    The definitions of $R\Gamma_\fg (X, \ZZ(n))$ and $R\Gamma_\Wc (X, \ZZ(n))$
    yield exact sequences
    \[ \begin{tikzcd}[row sep=1em,column sep=1em]
        &[-1em] H^{i-1}_c (G_\RR, X (\CC), \ZZ (n)) \ar{d} &[-1em] \\
        & H^i_\Wc (X, \ZZ(n)) \ar{d} \\
        H^i_c (X_\et, \ZZ(n)) \ar{r} & H^i_\fg (X, \ZZ(n)) \ar{r}\ar{d} & \Hom (H^{1-i} (X_\et, \ZZ^c (n)), \QQ) \ar{r} & H^{i+1}_c (X_\et, \ZZ(n)) \\
        & H^i_c (G_\RR, X (\CC), \ZZ (n))
      \end{tikzcd} \]
    If $i < 0$, then $H^i_c (X_\et, \ZZ(n)) = H^i_c (G_\RR, X (\CC), \ZZ (n)) = 0$.
    Moreover, $\Hom (H^{1-i} (X_\et, \ZZ^c (n)), \QQ) = 0$ for $i < 0$, since
    $H^{1-i} (X_\et, \ZZ^c (n))$ is finite $2$-torsion
    (Proposition~\ref{prop:motivic-cohomology-duality-consequences}). We conclude
    that $H^i_\Wc (X, \ZZ(n)) = H^i_\fg (X, \ZZ(n)) = 0$ for $i < 0$.
  \end{proof}
\end{lemma}

For the vanishing of $H^i_\Wc (X, \ZZ(n))$ for $i \gg 0$, we first establish
the following auxiliary result.

\begin{lemma}
  Let $d = \dim X$. For each prime $\ell$ and $i \ge 2d$ we have
  \begin{equation}
    \label{eqn:l-adic-completion-of-H-Wc}
    H^i_\Wc (X, \ZZ (n)) \otimes \ZZ_\ell =
    \widehat{H}_c^i (X [1/\ell]_\et, \ZZ_\ell (n)),
  \end{equation}
  where the right-hand side is defined via
  $\varprojlim_r \widehat{H}_c^i (X [1/\ell]_\et, \mu_{\ell^r}^{\otimes n})$.

  \begin{proof}
    Consider the commutative diagram with distinguished rows and columns
    \[ \begin{tikzcd}[font=\footnotesize]
        {[R\Gamma (X_\et, \ZZ^c(n)), \QQ[-2]]} \ar{r}{\widehat{\alpha}_{X,n}} \ar{d}{id} & R\widehat{\Gamma}_c (X_\et, \ZZ(n)) \ar{d} \ar{r} & R\widehat{\Gamma}_\fg (X, \ZZ(n)) \ar{r}\ar{d} & {[+1]} \ar{d}{id} \\
        {[R\Gamma (X_\et, \ZZ^c(n)), \QQ[-2]]} \ar{r}{\alpha_{X,n}} \ar{d} & R\Gamma_c (X_\et, \ZZ(n)) \ar{r} \ar{d}{\widehat{u}^*_\infty} & R\Gamma_\fg (X, \ZZ(n)) \ar{r} \ar{d}{\widehat{i}^*_\infty} & {[+1]} \ar{d} \\
        0 \ar{r}\ar{d} & R\widehat{\Gamma}_c (G_\RR, X(\CC), \ZZ(n)) \ar{r}{id} \ar{d} & R\widehat{\Gamma}_c (G_\RR, X(\CC), \ZZ(n)) \ar{r} \ar{d} & 0 \ar{d} \\
        {[R\Gamma (X_\et, \ZZ^c(n)), \QQ[-1]]} \ar{r}{\widehat{\alpha}_{X,n} [1]} & R\widehat{\Gamma}_c (X_\et, \ZZ(n)) [1] \ar{r} & R\widehat{\Gamma}_\fg (X, \ZZ(n)) [1] \ar{r} & {[+2]}
      \end{tikzcd} \]
    Here $\widehat{u}^*_\infty$ (resp. $\widehat{i}^*_\infty$) is defined as the
    composition of the canonical morphism $u^*_\infty$ (resp. $i^*_\infty$) with
    the projection to the Tate cohomology
    \[ \pi\colon R\Gamma_c (G_\RR, X(\CC), \ZZ(n)) \to
      R\widehat{\Gamma}_c (G_\RR, X(\CC), \ZZ(n)). \]
    By Proposition~\ref{prop:equivariant-coho-of-X(C)}, $H^i (\pi)$ is an
    isomorphism for $i \ge 2d-1$. Therefore, the five-lemma applied to
    \[ \begin{tikzcd}
        R\Gamma_\Wc (X, \ZZ(n)) \ar{r}\ar{d}{f} & R\Gamma_\fg (X, \ZZ(n)) \ar{r}{i^*_\infty}\ar{d}{id} & R\Gamma_c (G_\RR, X(\CC), \ZZ(n)) \ar{r}\ar{d}{\pi} & {[+1]}\ar{d}{f [1]} \\
        R\widehat{\Gamma}_\fg (X, \ZZ(n)) \ar{r} & R\Gamma_\fg (X, \ZZ(n)) \ar{r}{\widehat{i}^*_\infty} & R\widehat{\Gamma}_c (G_\RR, X(\CC), \ZZ(n)) \ar{r} & {[+1]}
      \end{tikzcd} \]
    shows that for $i \ge 2d$ holds
    \[ H^i_\Wc (X, \ZZ(n)) \cong \widehat{H}^i_\fg (X, \ZZ(n)). \]
    As in Corollary~\ref{cor:RGamma-fg-model-for-l-adic-cohomology}, we have for a
    prime $\ell$
    \[ \widehat{H}^i_\fg (X, \ZZ(n)) \otimes \ZZ_\ell \cong \widehat{H}^i_c (X[1/\ell]_\et, \ZZ_\ell (n)). \qedhere \]
  \end{proof}
\end{lemma}

\begin{corollary}
  One has $H^i_\Wc (X, \ZZ(n)) = 0$ for $i > 2d+1$.

  \begin{proof}
    It suffices to verify that $H^i_\Wc (X, \ZZ(n)) \otimes \ZZ_\ell = 0$
    for each prime $\ell$. Thanks to the isomorphism
    \eqref{eqn:l-adic-completion-of-H-Wc}, this reduces to
    $\widehat{H}_c^i (X [1/\ell]_\et, \ZZ_\ell (n)) = 0$ for $i > 2d+1$,
    which is true for reasons of cohomological dimension
    \cite[Expos\'{e}~X, Th\'{e}or\`{e}me~6.2]{SGA4}. We note that if $\ell = 2$ and
    $X (\RR) \ne \emptyset$, then the usual \'{e}tale cohomology has finite
    $2$-torsion in arbitrarily high degrees. It is important that we consider
    here the \emph{modified} cohomology with compact support
    $\widehat{H}_c^i (-)$. To obtain the corresponding statement, combine the
    arguments from \cite[Expos\'{e}~X]{SGA4} with the well-known computations of
    modified cohomology for number fields; cf. \cite[Chapter~II]{Milne-ADT} and
    \cite{Artin-Verdier-1964}, \cite{Mazur-1973}.
  \end{proof}
\end{corollary}

Summarizing the above observations, we obtain the following result.

\begin{proposition}
  \label{prop:RGammaWc-perfect}
  Conjecture $\mathbf{L}^c (X_\et,n)$ implies that $R\Gamma_\Wc (X, \ZZ(n))$
  for $n < 0$ is a perfect complex. More precisely,
  $H^i_\Wc (X, \ZZ(n))$ are finitely generated groups, and
  $H^i_\Wc (X, \ZZ(n)) = 0$ for $i \notin [0, 2\dim X + 1]$.
\end{proposition}

\subsection*{Rational coefficients}

\begin{proposition}
  There is a non-canonical splitting
  \[ R\Gamma_\Wc (X, \ZZ(n)) \otimes \QQ \cong
    \RHom (R\Gamma (X_\et, \ZZ^c (n)), \QQ) [-1] \oplus
    R\Gamma_c (G_\RR, X (\CC), \QQ (n)) [-1]. \]

  \begin{proof}
    The distinguished triangle defining $R\Gamma_\Wc (X, \ZZ(n))$ becomes after
    tensoring with $\QQ$
    \begin{multline*}
      R\Gamma_\Wc (X, \ZZ (n))\otimes \QQ \to
      R\Gamma_\fg (X, \ZZ(n))\otimes \QQ \xrightarrow{i_\infty^*\otimes\QQ = 0}
      R\Gamma_c (G_\RR, X (\CC), \ZZ(n))\otimes \QQ \\
      \to R\Gamma_\Wc (X, \ZZ (n))\otimes \QQ [1]
    \end{multline*}
    which yields a non-canonical splitting \cite[Chapitre~II,
      Corollaire~1.2.6]{Verdier-thesis}
    \[ R\Gamma_\Wc (X, \ZZ (n))\otimes \QQ \cong
      R\Gamma_\fg (X, \ZZ(n))\otimes \QQ \oplus
      R\Gamma_c (G_\RR, X (\CC), \ZZ(n)) [-1]\otimes \QQ, \]
    and we have already established in
    Proposition~\ref{prop:tensoring-RGammafg-with-Z/m-and-Q} that
    \[ R\Gamma_\fg (X, \ZZ (n)) \otimes \QQ \cong
      \RHom (R\Gamma (X_\et, \ZZ^c (n)), \QQ) [-1]. \qedhere \]
  \end{proof}
\end{proposition}


\section{Known cases of Conjecture $\mathbf{L}^c (X_\et,n)$}
\label{sec:known-cases-of-Lc-Xet-n}

Since the main constructions of this paper assume Conjecture $\mathbf{L}^c
  (X_\et,n)$, we relate it here to other conjectures about the finite generation
of \'{e}tale motivic cohomology formulated in the literature, and also describe
certain schemes $X$ for which $\mathbf{L}^c (X_\et,n)$ holds unconditionally.

\vspace{1em}

Instead of our $\mathbf{L}^c (X_\et, -)$, Flach and Morin state in \cite[\S
  3]{Flach-Morin-2018} a slightly different conjecture $\mathbf{L} (X_\et, -)$.
For proper regular $X$ of pure dimension $d$, the following is a reformulation
of $\mathbf{L} (X_\et, d - n)$ \cite[Conjecture~3.2,
  Lemma~3.3]{Flach-Morin-2018} in terms of $\ZZ^c (n)$.

\begin{conjecture}
  \label{equivalent-conjecture-1}
  For a proper regular arithmetic scheme $X$ and
  $n < 0$, the groups $H^i (X_\et, \ZZ^c (n))$ are finitely generated for
  $i \le -2n+1$.
\end{conjecture}

A more precise conjectural description of \'{e}tale motivic cohomology is
\cite[Conjecture~4.12]{Geisser-2017}, which can be written for $\ZZ^c (n)$ as
follows:

\begin{conjecture}
  \label{equivalent-conjecture-2}
  For a proper regular arithmetic scheme $X$ and $n < 0$, one has
  \[ H^i (X_\et, \ZZ^c (n)) = \begin{cases}
      \text{finitely generated}, & i \le -2n,     \\
      \text{finite},             & i = -2n + 1,   \\
      \text{cofinite type},      & i \ge -2n + 2.
    \end{cases} \]
\end{conjecture}

This is consistent with our $\mathbf{L}^c (X_\et, n)$.

\begin{proposition}
  \label{prop:equivalent-conjectures}
  Let $X$ be a proper regular arithmetic scheme and $n < 0$.
  Then Conjecture $\mathbf{L}^c (X_\et, n)$,
  Conjecture~\ref{equivalent-conjecture-1},
  and Conjecture~\ref{equivalent-conjecture-2} are equivalent.

  \begin{proof}
    For the implication
    $\text{Conjecture~\ref{equivalent-conjecture-1}} \Longrightarrow \mathbf{L}^c (X_\et, n)$,
    by \cite[Proposition~3.4]{Flach-Morin-2018},
    Conjecture~\ref{equivalent-conjecture-1} implies Artin--Verdier duality
    \[ H^i (X_\et, \ZZ (n)) \cong \Hom (H^{2-i} (X_\et, \ZZ^c (n)), \QQ/\ZZ)
      \text{ up to finite }2\text{-torsion}, \]
    hence $H^i (X_\et, \ZZ^c (n))$ is finite $2$-torsion for $i \ge 2$, and in
    particular for $i > -2n + 1$.

    The implication $\text{Conjecture~\ref{equivalent-conjecture-1}}
      \Longrightarrow \text{Conjecture~\ref{equivalent-conjecture-2}}$ is also
    established in \cite[Proposition~3.4]{Flach-Morin-2018}.
  \end{proof}
\end{proposition}

We now list some special cases where Conjecture $\mathbf{L}^c (X_\et,n)$ is
known, and therefore gives unconditional results. We follow \cite[\S
  5]{Morin-2014} very closely. For an arithmetic scheme $X$, we formulate the
following conjecture, which is the conjunction of $\mathbf{L}^c (X_\et,n)$ for
all $n < 0$.

\begin{conjecture}
  $\mathbf{L}^c (X_\et)$: the cohomology groups $H^i (X_\et, \ZZ^c (n))$ are
  finitely generated for all $i \in \ZZ$ and $n < 0$.
\end{conjecture}

This is similar to \cite[Definition~5.8]{Morin-2014}, with the only difference
that Morin also requires the finite generation of $H^i (X_\et, \ZZ^c (0))$ for
$i \le 0$. Conjecture $\mathbf{L}^c (X_\et)$ is known for number rings, and
also for certain varieties over finite fields. As in \cite{Soule-1984},
\cite{Geisser-2004}, and \cite{Morin-2014}, we consider the following class.

\begin{definition}
  Let $A (\FF_q)$ be the full subcategory of the category of smooth projective
  varieties over a finite field $\FF_q$ generated by products of curves and the
  following operations.
  \begin{enumerate}
    \item[1)] If $X$ and $Y$ lie in $A (\FF_q)$, then $X \sqcup Y$ lies
          $A (\FF_q)$.
    \item[2)] If $Y$ lies in $A (\FF_q)$ and there are morphisms $c\colon X\to Y$
          and $c'\colon Y\to X$ in the category of Chow motives such that
          $c'\circ c\colon X\to X$ is a multiplication by constant, then
          $X$ lies in $A (\FF_q)$.
    \item[3)] If $\FF_{q^m}/\FF_q$ is a finite extension and
          $X_{\FF_q^m} = X \times_{\Spec \FF_q} \Spec \FF_{q^m}$ lies in
          $A (\FF_{q^m})$, then $X$ lies in $A (\FF_q)$.
    \item[4)] If $X$ and $Y$ lie in $A (\FF_q)$, and $Y$ is a closed subscheme of
          $X$, then the blowup of $X$ along $Y$ lies in $A (\FF_q)$.
  \end{enumerate}
\end{definition}

The following is similar to \cite[Definition~5.9]{Morin-2014}.

\begin{definition}
  Let $\mathcal{L} (\ZZ)$ be the full subcategory of arithmetic schemes
  generated by the following objects:
  \begin{itemize}
    \item the empty scheme $\emptyset$,
    \item $\Spec \mathcal{O}_F$ for a number field $F$,
    \item varieties $X \in A (\FF_q)$ for any finite field $\FF_q$,
  \end{itemize}
  and the following operations.
  \begin{enumerate}
    \item[$\mathcal{L}$1)] Let $X$ be an arithmetic scheme, $Z \subset X$ a closed
          subscheme and $U \dfn X\setminus Z$ its open complement. If two of three
          schemes $X,Z,U$ lie in $\mathcal{L} (\ZZ)$, then the third also lies in
          $\mathcal{L} (\ZZ)$.

    \item[$\mathcal{L}$2)] A finite disjoint union
          $X = \coprod_{1 \le j \le p} X_j$ lies in $\mathcal{L} (\ZZ)$ if and only if
          each $X_j$ lies in $\mathcal{L} (\ZZ)$.

    \item[$\mathcal{L}$3)] If $V \to U$ is an affine bundle and $U$ lies in
          $\mathcal{L} (\ZZ)$, then $V$ also lies in $\mathcal{L} (\ZZ)$.

    \item[$\mathcal{L}$4)] If $\{ U_i \to X \}_{i \in I}$ is a finite surjective
          family of \'{e}tale morphisms such that each $U_{i_0,\ldots,i_p}$ lies in
          $\mathcal{L} (\ZZ)$, then $X$ also lies in $\mathcal{L} (\ZZ)$.
  \end{enumerate}
\end{definition}

\begin{proposition}
  Conjecture $\mathbf{L}^c (X_\et)$ holds for any arithmetic scheme
  $X \in \mathcal{L} (\ZZ)$.

  \begin{proof}
    See the argument in \cite[Proposition~5.10]{Morin-2014}.
  \end{proof}
\end{proposition}

Finally, we consider cellular schemes, as in \cite[\S 5.4]{Morin-2014}.

\begin{definition}
  Let $Y$ be a separated scheme of finite type over $\Spec k$ for a field
  $k$. We say that $Y$ \textbf{admits a cellular decomposition} if there exists
  a filtration of $Y$ by reduced closed subschemes
  $$Y^{red} = Y_N \supseteq Y_{N-1} \supseteq \cdots \supseteq Y_{-1} = \emptyset$$
  such that $Y_i\setminus Y_{i-1} \cong \AA^{r_i}_k$ is isomorphic to an affine
  space over $k$.

  We say that $Y$ is \textbf{geometrically cellular} if $Y_{\overline{k}} = Y
    \times_{\Spec k} \Spec \overline{k}$ admits a cellular decomposition. This is
  equivalent to the existence of a finite Galois extension $k'/k$ such that
  $Y_{k'}$ admits a cellular decomposition.

  Finally, given an $S$-scheme $X \to S$ that is separated and of finite type, we
  say that $X$ is \textbf{geometrically cellular over $S$} if for each $s \in S$
  the corresponding fiber $X_s$ is geometrically cellular.
\end{definition}

\begin{proposition}
  Let $Y$ be a separated scheme of finite type over $\Spec \FF_q$.
  If $Y$ is geometrically cellular, then $X \in \mathcal{L} (\ZZ)$,
  and in particular Conjecture $\mathbf{L}^c (Y_\et)$ holds.

  If $X \to \Spec \mathcal{O}_F$ is a flat, separated scheme of finite type over
  the ring of integers of a number field, and $X$ is geometrically cellular over
  $\mathcal{O}_F$, then $X \in \mathcal{L} (\ZZ)$, and in particular
  $\mathbf{L}^c (X_\et)$ holds.
\end{proposition}

For a proof, we refer to \cite[Proposition~5.14]{Morin-2014}.


\section{Comparison with the complex of Flach and Morin}
\label{sec:comparison-with-FM}

This paper is based on the ideas of Flach and Morin \cite{Flach-Morin-2018},
who gave a similar construction of Weil-\'{e}tale cohomology $R\Gamma_\Wc (X,
  \ZZ(n))$ for a \emph{proper and regular} arithmetic scheme $X$, and for
\emph{any integer weight} $n \in \ZZ$. In this section, we will go through the
definitions of \cite{Flach-Morin-2018}, to verify the following claim.

\begin{proposition}
  \label{prop:comparison-with-FM}
  Let $X$ be a proper, regular arithmetic scheme, and $n < 0$. Assume
  Conjecture $\mathbf{L}^c (X_\et, n)$. Then the Weil-\'{e}tale complex
  $R\Gamma_\Wc (X, \ZZ(n))$ defined above in {\rm \S\ref{sec:RGamma-Wc}}
  is isomorphic to the corresponding complex defined in
    {\rm \cite{Flach-Morin-2018}}.
\end{proposition}

From now on we tacitly assume Conjecture $\mathbf{L}^c (X_\et, n)$, which is
also equivalent to the assumptions on motivic cohomology in
\cite{Flach-Morin-2018} (see Proposition~\ref{prop:equivalent-conjectures}).
Flach and Morin consider the case of a proper and regular arithmetic scheme $X$
of equal dimension $d$. In this case, we follow \cite[Remark
  3.11]{Flach-Morin-2018} to reformulate their constructions in terms of
complexes $\ZZ^c (n)$.

Moreover, they work with the Artin--Verdier \'{e}tale topos $\overline{X}_\et$,
whose definition and basic properties can be found in \cite[\S
  6]{Flach-Morin-2018}. They consider a morphism
\[ \overline{\alpha}_{X,n}\colon
  \RHom (R\Gamma (X, \ZZ^c (n)), \QQ [-2]) \to
  R\Gamma (\overline{X}_\et, \ZZ (n)), \]
defined in a similar way to our $\alpha_{X,n}$ (Definition~\ref{def:RGamma-fg})
using a duality similar to our Theorem~\ref{theorem-I}.

The notation in \cite{Flach-Morin-2018} and in this paper is intentionally the
same for various objects and morphisms. However, in this section we will write,
for example, $\overline{\alpha}_{X,n}$ to denote the morphism of Flach and
Morin, to distinguish it from our $\alpha_{X,n}$, etc. An overline indicates
that the corresponding thing comes from \cite{Flach-Morin-2018} and has
something to do with the Artin--Verdier \'{e}tale topos.

\begin{lemma}
  The square
  \begin{equation}
    \label{eqn:alpha-vs-alpha-bar-square}
    \begin{tikzcd}
      \RHom (R\Gamma (X, \ZZ^c (n)), \QQ [-2]) \ar{r}{\overline{\alpha}_{X,n}}\ar{d}{id} & R\Gamma (\overline{X}_\et, \ZZ(n))\ar{d} \\
      \RHom (R\Gamma (X, \ZZ^c (n)), \QQ [-2]) \ar{r}{\alpha_{X,n}} & R\Gamma (X_\et, \ZZ(n))
    \end{tikzcd}
  \end{equation}
  commutes.

  \begin{proof}
    We recall from Remark~\ref{rmk:alpha-X-n-determined-by-cohomology} that
    $\alpha_{X,n}$ is determined by the maps at the level of cohomology
    $H^i (\alpha_{X,n})$. The same is true for $\overline{\alpha}_{X,n}$, for
    the same reasons. Now \cite[Theorem~3.5]{Flach-Morin-2018} defines
    \begin{multline*}
      H^i (\overline{\alpha}_{X,n})\colon
      \Hom (H^{2-i} (X, \ZZ^c (n)), \QQ) \xrightarrow{\cong}
      \Hom (H^{2-i} (\overline{X}_\et, \ZZ^c (n)), \QQ) \to \\
      \Hom (H^{2-i} (\overline{X}_\et, \ZZ^c (n)), \QQ/\ZZ) \xleftarrow{\cong}
      H^i (\overline{X}_\et, \ZZ (n)),
    \end{multline*}
    where the last isomorphism is the duality
    \cite[Corollary~6.26]{Flach-Morin-2018}. Similarly, our morphism
    $\alpha_{X,n}$ gives
    \begin{multline*}
      H^i (\alpha_{X,n})\colon
      \Hom (H^{2-i} (X, \ZZ^c (n)), \QQ) \xrightarrow{\cong}
      \Hom (H^{2-i} (X_\et, \ZZ^c (n)), \QQ) \to \\
      \Hom (H^{2-i} (X_\et, \ZZ^c (n)), \QQ/\ZZ) \xleftarrow{\cong}
      \widehat{H}^i_c (X_\et, \ZZ (n)) \to
      H^i (X_\et, \ZZ (n)).
    \end{multline*}

    The groups $\widehat{H}^i_c (X_\et, \ZZ (n))$ and $H^i (\overline{X}_\et, \ZZ
      (n))$ are different, but the duality in terms of $H^i (\overline{X}_\et, \ZZ
      (n))$ is compatible with the duality in terms of $\widehat{H}^i_c (X_\et, \ZZ
      (n))$ (see \cite[Theorem~6.24]{Flach-Morin-2018}): we have a commutative
    diagram
    \[ \begin{tikzcd}
        R\widehat{\Gamma}_c (X_\et, \ZZ/m\ZZ(n)) \ar{d}\ar{r}{\cong} & \RHom (R\Gamma (X_\et, \ZZ/m\ZZ^c (n)), \QQ/\ZZ [-2])\ar{d} \\
        R\Gamma (\overline{X}_\et, \ZZ/m\ZZ(n)) \ar{r}{\cong} & \RHom (R\Gamma (\overline{X}_\et, \ZZ/m\ZZ^c (n)), \QQ/\ZZ [-2])
      \end{tikzcd} \]
    and the diagram
    \[ \begin{tikzcd}
        R\widehat{\Gamma}_c (X_\et, \ZZ(n)) \ar{r}\ar{d} & R\Gamma (X_\et, \ZZ(n)) \\
        R\Gamma (\overline{X}_\et, \ZZ(n) \ar{ur}
      \end{tikzcd} \]
    commutes as well. We see that the diagram we are interested in commutes:
    \[ \begin{tikzcd}[column sep=1.25em]
        \Hom (H^{2-i} (X, \ZZ^c (n)), \QQ) \ar{r}\ar{d}{id} \ar[bend left=15]{rrr}{H^i (\overline{\alpha}_{X,n})} & H^{2-i} (\overline{X}_\et, \ZZ^c (n))^D & & H^i (\overline{X}_\et, \ZZ(n))\ar{d} \ar{ll}[swap]{\cong} \\
        \Hom (H^{2-i} (X, \ZZ^c (n)), \QQ) \ar{r} \ar[bend right=15]{rrr}[swap]{H^i (\alpha_{X,n})} & H^{2-i} (X_\et, \ZZ^c (n))^D \ar[dashed]{u} & \widehat{H}^i_c (X_\et, \ZZ(n)) \ar{l}[swap]{\cong}\ar{r}\ar[dashed]{ur} & H^i (X_\et, \ZZ(n))
      \end{tikzcd} \]
    For brevity, $\Hom (A,\QQ/\ZZ)$ is denoted here by $A^D$.
  \end{proof}
\end{lemma}

Taking the cones of $\overline{\alpha}_{X,n}$ and $\alpha_{X,n}$, we obtain
respectively the complex $R\Gamma_W (\overline{X}, \ZZ (n))$ of Flach and Morin
\cite[Definition~3.6]{Flach-Morin-2018} and our complex $R\Gamma_\fg (X,
  \ZZ(n))$ (Definition~\ref{def:RGamma-fg} above).

The square \eqref{eqn:alpha-vs-alpha-bar-square} induces the following diagram
with distinguished rows and columns (cf.
\cite[Proposition~1.4.6]{Neeman-2001}):
\begin{equation}
  \label{eqn:cones-of-alphas}
  \begin{tikzcd}[column sep=1em,font=\small]
    {[R\Gamma (X, \ZZ^c (n)), \QQ [-2]]} \ar{r}{\overline{\alpha}_{X,n}}\ar{d}{id} & R\Gamma (\overline{X}_\et, \ZZ(n))\ar{d}\ar{r}{f} & R\Gamma_W (\overline{X}, \ZZ(n)) \ar{r}\ar[dashed]{d} & {[-1]}\ar{d}{id} \\
    {[R\Gamma (X, \ZZ^c (n)), \QQ [-2]]} \ar{r}{\alpha_{X,n}}\ar{d} & R\Gamma (X_\et, \ZZ(n)) \ar{r}{g}\ar{d} & R\Gamma_\fg (X, \ZZ(n)) \ar{r}\ar{d} & {[-1]}\ar{d} \\
    0\ar{r}\ar{d} & R\Gamma (X (\RR), \tau_{\ge n+1} R \widehat{\pi}_* \ZZ (n)) \ar{r}{id}\ar{d} & R\Gamma (X (\RR), \tau_{\ge n+1} R \widehat{\pi}_* \ZZ (n)) \ar{r}\ar{d} & 0\ar{d} \\
    {[R\Gamma (X, \ZZ^c (n)), \QQ [-1]]} \ar{r} & R\Gamma (\overline{X}_\et, \ZZ(n))[1]\ar{r}{f[1]} & R\Gamma_W (\overline{X}, \ZZ(n))[1] \ar{r} & {[0]}
  \end{tikzcd}
\end{equation}

Then \cite[Definition~3.23]{Flach-Morin-2018} considers a morphism
$\overline{u}^*_\infty$ defined via
\begin{equation}
  \label{eqn:definition-of-u-bar-infty}
  \begin{tikzcd}[column sep=1em]
    R\Gamma (\overline{X}_\et, \ZZ(n)) \ar{r}\ar[dashed]{d}[swap]{\exists}{\overline{u}_\infty^*} & R\Gamma (X_\et, \ZZ(n)) \ar{r}\ar{d}{u_\infty^*} & R\Gamma (X(\RR), \tau_{\ge n+1} R \widehat{\pi}_* \ZZ (n)) \ar{r}\ar{d}{id} & {[+1]} \ar[dashed]{d}{\overline{u}_\infty^* [1]} \\
    R\Gamma_W (X_\infty, \ZZ (n)) \ar{r} & R\Gamma (G_\RR, X (\CC), \ZZ (n)) \ar{r} & R\Gamma (X(\RR), \tau_{\ge n+1} R \widehat{\pi}_* \ZZ (n)) \ar{r} & {[+1]}
  \end{tikzcd}
\end{equation}
Here the complex $R\Gamma_W (X_\infty, \ZZ(n))$ is \emph{defined} via the bottom
triangle.

Then \cite[Proposition~3.24]{Flach-Morin-2018} and our
Proposition~\ref{prop:uniqueness-of-i-infty} above establish the existence and
uniqueness of morphisms $\overline{\iota}_\infty^*$ and $i_\infty^*$ which make
the triangles below commutative, and then the Weil-\'{e}tale complexes are
defined as mapping fibers of $\overline{\iota}_\infty^*$ and $i_\infty^*$:

\[ \begin{tikzcd}[column sep=1em]
    R\Gamma_\Wc (\overline{X}, \ZZ (n))\ar{d} & & R\Gamma_\Wc (X, \ZZ (n))\ar{d} \\
    R\Gamma_W (\overline{X}, \ZZ(n)) \ar[dashed]{d}{\overline{\iota}_\infty^*} & R\Gamma (\overline{X}_\et, \ZZ(n)) \ar{dl}{\overline{u}_\infty^*}\ar{l}[swap]{f} & R\Gamma_\fg (X, \ZZ(n))  \ar[dashed]{d}{\iota_\infty^*} & R\Gamma (X_\et, \ZZ(n)) \ar{dl}{u_\infty^*}\ar{l}[swap]{g} \\
    R\Gamma_W (X_\infty, \ZZ (n))\ar{d} & & R\Gamma (G_\RR, X (\CC), \ZZ (n))\ar{d} \\
    R\Gamma_\Wc (\overline{X}, \ZZ (n))[1] & & R\Gamma_\Wc (X, \ZZ (n))[1] \\
  \end{tikzcd} \]

In order to compare the two resulting complexes, we note that
$\overline{u}_\infty^*$ is only defined via
\eqref{eqn:definition-of-u-bar-infty}, so in the diagram below from
Figure~\ref{fig:comparison-with-FM}, we can first choose
$\overline{\iota}_\infty^*$ such that the front face gives a morphism of
triangles. Then we can \emph{declare} $\overline{u}_\infty^*$ to be the
composition $\overline{\iota}_\infty^* \circ f$. In this way everything
commutes, and we see that $R\Gamma_\Wc (\overline{X}, \ZZ(n)) \cong R\Gamma_\Wc
  (X, \ZZ(n))$.

\vspace{1em}

This concludes the proof of Proposition~\ref{prop:comparison-with-FM}. \qed

\begin{landscape}
  \begin{figure}
    \[ \begin{tikzcd}[column sep=0.5em, row sep=3em,font=\small]
        R\Gamma_\Wc (\overline{X}, \ZZ(n)) \ar{dd}\ar[dashed]{rr}{\cong} &[-1.5em] &[-1.5em] R\Gamma_\Wc (X, \ZZ(n))\ar{rr} &[-1.5em] &[-1.5em] 0 \ar{rr} &[-2.5em] & {[+1]} \\
        & R\Gamma (\overline{X}_\et, \ZZ(n)) \ar{rr}\ar[near end]{dddl}{\overline{u}_\infty^*}\ar{dl}[swap]{f} & & R\Gamma (X_\et, \ZZ(n)) \ar[near end]{dddl}{u_\infty^*}\ar{dl}[swap]{g}\ar{rr} && R\Gamma (X (\RR), \tau_{\ge n+1} R\widehat{\pi}_* \ZZ (n))\ar{dl}[swap]{id}\ar[near end]{dddl}{id}\ar{rr} &&[1em] {[+1]}\ar{dddl}{\overline{u}_\infty^* [1]} \ar{dl}[swap]{f[1]} \\
        R\Gamma_W (\overline{X}, \ZZ(n)) \ar[dashed]{dd}[swap]{\overline{\iota}_\infty^*}\ar[crossing over]{rr} & & R\Gamma_\fg (X, \ZZ(n)) \ar{dd}[swap]{i_\infty^*}\ar[crossing over]{rr}\ar[<-,crossing over]{uu} && R\Gamma (X (\RR), \tau_{\ge n+1} R\widehat{\pi}_* \ZZ (n)) \ar{dd}[swap]{id}\ar[<-,crossing over]{uu} \ar[crossing over]{rr} && {[+1]}\ar[<-,crossing over]{uu}\ar[dashed]{dd}[swap]{\overline{\iota}_\infty^* [1]} \\
        \\
        R\Gamma_W (X_\infty, \ZZ (n)) \ar{rr}\ar{dd} & & R\Gamma (G_\RR, X (\CC), \ZZ (n)) \ar{rr}\ar{dd} && R\Gamma (X (\RR), \tau_{\ge n+1} R\widehat{\pi}_* \ZZ (n)) \ar{rr}\ar{dd} && {[+1]}\ar{dd} \\
        \\
        R\Gamma_\Wc (\overline{X}, \ZZ(n))[1] \ar[dashed]{rr}{\cong} && R\Gamma_\Wc (X, \ZZ(n))[1]\ar{rr} && 0 \ar{rr} && {[+2]}
      \end{tikzcd} \]

    \caption{Comparison of the Weil-\'{e}tale complexes from
    \cite{Flach-Morin-2018} and this paper, denoted
    $R\Gamma_\Wc (\overline{X}, \ZZ(n))$ and $R\Gamma_\Wc (X, \ZZ(n))$
    respectively. The top face of the prism comes from
    \eqref{eqn:cones-of-alphas}. The arrow $\overline{\iota}_\infty^*$ is
    chosen so that the front face is commutative. Then set
    $\overline{u}_\infty^* = \overline{\iota}_\infty^* \circ f$ so that the
    back face is commutative and corresponds to
    \eqref{eqn:definition-of-u-bar-infty}.}
    \label{fig:comparison-with-FM}
  \end{figure}
\end{landscape}


\pagebreak
\appendix
\section{Some homological algebra}
\label{app:homological-algebra}

This appendix contains some basic results about the derived category of abelian
groups $\DZ$ which are used throughout the text. The following lemmas are
isolated from the proofs in \cite{Flach-Morin-2018}, with some modifications to
treat the $2$-torsion.

First, recall that every complex of abelian groups $A^\bullet$ (not necessarily
bounded) is quasi-isomorphic to its cohomology:
\begin{multline*}
  A^\bullet \cong \coprod_{i\in \ZZ} H^i (A^\bullet) [-i] \cong
  \prod_{i\in \ZZ} H^i (A^\bullet) [-i] \\
  = \Bigl(\cdots \to H^{i-1} (A^\bullet) \xrightarrow{0} H^i (A^\bullet)
  \xrightarrow{0} H^{i+1} (A^\bullet) \to \cdots\Bigr).
\end{multline*}
Here
$\coprod_{i\in \ZZ} H^i (A^\bullet) [-i] \cong \prod_{i\in \ZZ} H^i (A^\bullet) [-i]$
is the complex that has $H^i (A^\bullet)$ in $i$-th degree, which serves as both
product and coproduct of complexes $H^i (A^\bullet) [-i]$
concentrated in $i$-th degree. This gives us a useful expression for morphisms
in the derived category: since $\Hom_\DZ (A,B [i]) \cong \Ext_\ZZ^i (A,B)$,
and $\Ext_\ZZ^i (A,B) = 0$ for $i > 1$, we obtain
\begin{align}
  \notag \Hom_\DZ (A^\bullet, B^\bullet) & \cong \Hom_\DZ (\coprod_{i\in\ZZ} H^i (A^\bullet) [-i], \prod_{j\in\ZZ} H^j (B^\bullet) [-j])                                       \\
  \notag                                 & \cong \prod_{i\in \ZZ} \prod_{j\in \ZZ} \Hom_\DZ (H^i (A^\bullet), H^j (B^\bullet) [i-j])                                           \\
  \notag                                 & \cong \prod_{i\in \ZZ} \left(\Hom (H^i (A^\bullet), H^i (B^\bullet)) \oplus \Ext (H^i (A^\bullet), H^{i-1} (B^\bullet))\right)      \\
  \label{eqn:morphisms-in-D(Z)}          & \cong \prod_{i\in \ZZ} \Hom (H^i (A^\bullet), H^i (B^\bullet)) \oplus \prod_{i\in \ZZ} \Ext (H^i (A^\bullet), H^{i-1} (B^\bullet)).
\end{align}

\begin{lemma}
  \label{lemma:morphisms-inDAb-not-divisible}
  ~

  \begin{enumerate}
    \item[$1)$] If $C^\bullet$ and $C'^\bullet$ are almost perfect in the sense of
          Definition~{\rm\ref{dfn:almost-of-(co)finite-type}}, then the group
          $\Hom_\DZ (C^\bullet, C'^\bullet)$ has no nontrivial
          divisible subgroups.

    \item[$2)$] If $A^\bullet$ is a complex such that $H^i (A^\bullet)$ are
          finite-dimensional $\QQ$-vector spaces and $C^\bullet$ is a complex such
          that $H^i (C^\bullet)$ are finitely generated abelian groups, then the group
          $\Hom_\DZ (A^\bullet, C^\bullet)$ is divisible.
  \end{enumerate}

  \begin{proof}
    In 1), if $C^\bullet$ and $C'^\bullet$ are almost perfect, then
    $\Hom (H^i (C^\bullet), H^i (C'^\bullet))$ are finitely generated groups,
    $2$-torsion for $i \gg 0$. Writing
    $H^i (C^\bullet) \cong \ZZ^{\oplus r} \oplus G$,
    $H^{i - 1} (C'^\bullet) \cong \ZZ^{\oplus r'} \oplus G'$ for some $r, r'$
    and finite groups $G, G'$, we calculate that
    \[
      \Ext (\ZZ^{\oplus r} \oplus G, \ZZ^{\oplus r'} \oplus G') \cong
      \underbrace{\Ext (G, \ZZ)}_{\cong G}{}^{\oplus r'} \oplus \Ext (G, G')
    \]
    are finite groups, $2$-torsion for $i \gg 0$. It follows from
    \eqref{eqn:morphisms-in-D(Z)} that $\Hom_\DZ (C^\bullet, C'^\bullet)$ is a sum
    of a finitely generated group and a $2$-torsion group, so it cannot have
    nontrivial divisible subgroups.

    Similarly in 2), under our assumption, $\Hom (H^i (A^\bullet), H^i (C^\bullet))
      = 0$ for all $i$, and the calculation
    \[
      \Ext (\QQ^{\oplus r}, \ZZ^{\oplus s} \oplus G) \cong
      \Ext (\QQ, \ZZ)^{\oplus r s} \oplus
      \underbrace{\Ext (\QQ, G)}_{= 0}{}^{\oplus r}
    \]
    shows that $\Hom_\DZ (A^\bullet, C^\bullet)$ is a direct product of divisible
    groups $\Ext (\QQ, \ZZ)$, hence divisible.
  \end{proof}
\end{lemma}

Recall that Verdier's axiom (TR1) states that every morphism $v\colon A^\bullet
  \to B^\bullet$ can be completed to a distinguished triangle $A^\bullet
  \xrightarrow{u} B^\bullet \xrightarrow{v} C^\bullet \xrightarrow{w} A^\bullet
  [1]$. Axiom (TR3) states that for every commutative diagram with distinguished
rows
\begin{equation}
  \label{eqn:TR3-input}
  \begin{tikzcd}
    A^\bullet\ar{r}{u}\ar{d}{f} & B^\bullet\ar{r}{v}\ar{d}{g} & C^\bullet\ar{r}{w} & A^\bullet [1] \\
    A'^\bullet\ar{r}{u'} & B'^\bullet\ar{r}{v'} & C'^\bullet\ar{r}{w'} & A'^\bullet [1]
  \end{tikzcd}
\end{equation}
there exists some $h\colon C^\bullet \to C'^\bullet$, which gives a morphism of
distinguished triangles
\begin{equation}
  \label{eqn:TR3-output}
  \begin{tikzcd}
    A^\bullet\ar{r}{u}\ar{d}{f} & B^\bullet\ar{r}{v}\ar{d}{g} & C^\bullet\ar{r}{w}\ar[dashed]{d}{\exists h} & A^\bullet [1]\ar{d}{f [1]} \\
    A'^\bullet\ar{r}{u'} & B'^\bullet\ar{r}{v'} & C'^\bullet\ar{r}{w'} & A'^\bullet [1]
  \end{tikzcd}
\end{equation}

The cone $C^\bullet$ in (TR1) and the morphism $h$ in (TR3) are neither unique
nor canonical. Two different cones of the same morphism are necessarily
isomorphic, but the isomorphism between them is not unique, because it is
provided by (TR3). Let us recall a useful argument showing that things are
well-defined in some special cases.

\begin{lemma}[{$\approx$\cite[Proposition~1.1.9, Corollaire~1.1.10]{Beilinson-Bernstein-Deligne}}]
  \label{lemma:TR3-TR1-with-uniqueness-general-statement}

  Consider the derived category $\mathbf{D} (\mathcal{A})$ of an abelian category
  $\mathcal{A}$.

  \begin{enumerate}
    \item[$1)$] For a commutative diagram \eqref{eqn:TR3-input}, assume that the
          homomorphism of abelian groups
          \[ w^*\colon \Hom_{\mathbf{D} (\mathcal{A})} (A^\bullet [1], C'^\bullet) \to
            \Hom_{\mathbf{D} (\mathcal{A})} (C^\bullet, C'^\bullet) \]
          induced by $w$ is trivial. Then there exists a unique morphism $h\colon
            C^\bullet \to C'^\bullet$ that gives a morphism of triangles
          \eqref{eqn:TR3-output}.

    \item[$2)$] For a distinguished triangle
          $A^\bullet \xrightarrow{u} B^\bullet \xrightarrow{v} C^\bullet \xrightarrow{w} A^\bullet[1]$,
          assume that for any other cone $C'^\bullet$ of $u$ the morphism $w^*$ is
          trivial. Then the cone of $u$ is unique up to a unique isomorphism.
  \end{enumerate}

  \begin{proof}
    In 1), applying $\Hom_{\mathbf{D} (\mathcal{A})} (-, C'^\bullet)$ to the
    first distinguished triangle, we obtain an exact sequence of abelian groups
    \[ \Hom_{\mathbf{D} (\mathcal{A})} (A^\bullet [1], C'^\bullet) \xrightarrow{w^*}
      \Hom_{\mathbf{D} (\mathcal{A})} (C^\bullet, C'^\bullet) \xrightarrow{v^*}
      \Hom_{\mathbf{D} (\mathcal{A})} (B^\bullet, C'^\bullet). \]
    If $w^* = 0$, we conclude that $v^*$ is a monomorphism. This implies that there
    is a unique morphism $h$ such that $h\circ v = v'\circ g$. Now in 2), if
    $C^\bullet$ and $C'^\bullet$ are two different cones of $u$, we have a
    commutative diagram
    \[ \begin{tikzcd}
        A^\bullet\ar{r}{u}\ar{d}{id} & B^\bullet\ar{r}{v}\ar{d}{id} & C^\bullet\ar{r}{w}\ar[dashed]{d} & A^\bullet [1]\ar{d}{id} \\
        A^\bullet\ar{r}{u'} & B^\bullet\ar{r}{v'} & C'^\bullet\ar{r}{w'} & A^\bullet [1]
      \end{tikzcd} \]
    By the triangulated five-lemma, the dashed arrow is an isomorphism, and it is
    unique thanks to part 1).
  \end{proof}
\end{lemma}

Here is a special case that we need.

\begin{corollary}
  \label{cor:TR3-TR1-with-uniqueness}
  Consider the derived category $\DZ$.

  \begin{enumerate}
    \item[$1)$] Suppose we have a commutative diagram with distinguished rows
          \eqref{eqn:TR3-input}, where $A^\bullet$ is a complex such that
          $H^i (A^\bullet)$ are finite-dimensional $\QQ$-vector spaces and
          $C^\bullet$, $C'^\bullet$ are almost perfect complexes in the sense of
          Definition~{\rm\ref{dfn:almost-of-(co)finite-type}}. Then there exists a
          unique morphism ${h\colon C^\bullet \to C'^\bullet}$ which gives a morphism
          of triangles \eqref{eqn:TR3-output}.

    \item[$2)$] For a distinguished triangle
          $$A^\bullet \xrightarrow{u} B^\bullet \xrightarrow{v} C^\bullet \xrightarrow{w} A^\bullet[1]$$
          assume that $A^\bullet$ is a complex such that $H^i (A^\bullet)$ are
          finite-dimensional $\QQ$-vector spaces and $C^\bullet$ is an almost perfect
          complex. Then the cone of $u$ is unique up to a unique isomorphism.
  \end{enumerate}

  \begin{proof}
    In this situation, by Lemma~\ref{lemma:morphisms-inDAb-not-divisible}, the
    group $\Hom_\DZ (C^\bullet, C'^\bullet)$ has no nontrivial divisible
    subgroups, and $\Hom_\DZ (A^\bullet [1], C'^\bullet)$ is divisible. This
    means that there are no nontrivial homomorphisms
    $\Hom_\DZ (A^\bullet [1], C'^\bullet) \to \Hom_\DZ (C^\bullet, C'^\bullet)$,
    and we can apply
    Lemma~\ref{lemma:TR3-TR1-with-uniqueness-general-statement}.
  \end{proof}
\end{corollary}

\begin{lemma}
  \label{lemma:torsion-morphisms-in-D(Z)}
  Suppose that $A^\bullet$ and $B^\bullet$ are almost of cofinite type in the
  sense of Definition~{\rm\ref{dfn:almost-of-(co)finite-type}}. Then a morphism
  $f\colon A^\bullet\to B^\bullet$ is torsion (i.e. a torsion element in the
  group $\Hom_\DZ (A^\bullet, B^\bullet)$, i.e.  $f\otimes \mathbb{Q} = 0$) if
  and only if the morphisms $H^i (f)\colon H^i (A^\bullet) \to H^i (B^\bullet)$
  are torsion; that is, they are trivial on the maximal divisible subgroups:
  $$(H^i (f)_\div\colon H^i (A^\bullet)_\div \to H^i (B^\bullet)_\div) = 0.$$

  \begin{proof}
    We may write $H^i (A^\bullet) \cong (\QQ/\ZZ)^{\oplus r} \oplus G$
    and
    $H^{i-1} (B^\bullet) \cong (\QQ/\ZZ)^{\oplus s} \oplus H$ for some
    $r, s$ and some finite groups $G, H$. Now
    \[
      \Ext ((\QQ/\ZZ)^{\oplus r} \oplus G, (\QQ/\ZZ)^{\oplus s} \oplus H) \cong
      \underbrace{\Ext (\QQ/\ZZ, H)}_{\cong H}{}^{\oplus r} \oplus \Ext (G, H)
    \]
    is a finite group. It follows that tensoring \eqref{eqn:morphisms-in-D(Z)} with
    $\QQ$ kills $\prod_{i\in \ZZ} \Ext (H^i (A^\bullet), H^{i-1} (B^\bullet))$ and
    gives an isomorphism
    \begin{align*}
      \Hom_\DZ (A^\bullet, B^\bullet)\otimes \mathbb{Q} & \cong
      \prod_{i\in \mathbb{Z}} \Hom (H^i (A^\bullet), H^i (B^\bullet)) \otimes \mathbb{Q},                   \\
      f\otimes\QQ                                       & \mapsto (H^i (f) \otimes \QQ)_{i\in\ZZ}. \qedhere
    \end{align*}
  \end{proof}
\end{lemma}

\begin{lemma}
  \label{lemma:morphisms-in-DAb-between-cplx-of-Q-vs-and-almost-cofinite-type-cplx}
  If $A^\bullet$ is a complex of $\QQ$-vector spaces and $B^\bullet$ is a
  complex almost of cofinite type in the sense of
  Definition~{\rm\ref{dfn:almost-of-(co)finite-type}}, then there is an
  isomorphism of abelian groups
  \begin{align*}
    \Hom_\DZ (A^\bullet, B^\bullet) & \xrightarrow{\cong}
    \prod_{i\in \ZZ} \Hom (H^i (A^\bullet), H^i (B^\bullet)),       \\
    f                               & \mapsto (H^i (f))_{i\in \ZZ}.
  \end{align*}

  \begin{proof}
    If $H^i (A^\bullet)$ are
    $\QQ$-vector spaces and $H^{i-1} (B^\bullet)$ are groups of cofinite type,
    then the term $\Ext (H^i (A^\bullet), H^{i-1} (B^\bullet))$ in the formula
    \eqref{eqn:morphisms-in-D(Z)} vanishes by calculations similar to the above,
    as $\Ext (\QQ, \QQ/\ZZ) = \Ext (\QQ, G) = 0$ for finite $G$.
  \end{proof}
\end{lemma}


\section{Cohomology with compact support}
\label{app:modified-cohomology-with-compact-support}

For any arithmetic scheme $f\colon X\to \Spec \ZZ$ there exists a
\textbf{Nagata compactification}
\cite{Conrad-Deligne-Nagata,Conrad-Deligne-Nagata-erratum} (see also
\cite[Expos\'{e}~XVII]{SGA4})
\[ \begin{tikzcd}
    X \ar[hookrightarrow]{rr}{j}\ar{dr}[swap]{f} & & \mathfrak{X} \ar{dl}{g} \\
    & \Spec \ZZ
  \end{tikzcd} \]
where $j$ is an open immersion and $g$ is a proper morphism.

\begin{definition}
  Let $X$ be an arithmetic scheme and let $\mathcal{F}$ be an abelian torsion
  sheaf on $X_\et$. Then one defines the
  \textbf{cohomology with compact support} of $\mathcal{F}$ via the complex
  \[
    R\Gamma_c (X_\et, \mathcal{F}) \dfn
    R\Gamma (\mathfrak{X}_\text{\it \'et}, j_! \mathcal{F}).
  \]
\end{definition}

For torsion sheaves, this does not depend on the choice of $j\colon X
  \hookrightarrow \mathfrak{X}$, but here we would like to fix this choice in
order to compare cohomology with compact support on $X_\et$ with the singular
cohomology with compact support on $X (\CC)$.

\subsection*{Comparison with the analytic cohomology}

\begin{definition}
  Given a Nagata compactification $j\colon X\hookrightarrow \mathfrak{X}$,
  we consider the corresponding open immersion
  $j (\CC)\colon X (\CC) \to \mathfrak{X} (\CC)$,
  and for a sheaf $\mathcal{F}$ on $X (\CC)$ we define
  \[ \Gamma_c (X (\CC), \mathcal{F}) \dfn
    \Gamma (\mathfrak{X} (\CC), j (\CC)_! \mathcal{F}). \]
  Similarly, for a $G_\RR$-equivariant sheaf on $X (\CC)$ we define
  \[ \Gamma_c (G_\RR, X (\CC), \mathcal{F}) \dfn
    \Gamma (G_\RR, \mathfrak{X} (\CC), j (\CC)_! \mathcal{F}). \]
\end{definition}

The canonical reference for the comparison between \'{e}tale and singular
cohomology is \cite[Expos\'{e}~XI, \S 4]{SGA4}, so we borrow some definitions
and notations from there. Let $X$ be an arithmetic scheme.

\begin{enumerate}
  \item The base change from $\Spec \ZZ$ to $\Spec \CC$ gives us a morphism of sites
        $$\gamma\colon X_{\CC,\text{\it \'{e}t}} \to X_\et.$$

  \item Let $X_\text{\it cl}$ be the site of \'{e}tale maps $f\colon U\to X (\CC)$. A
        covering family in $X_\text{\it cl}$ is a family of maps $\{ U_i \to U \}$ such
        that $U$ is the union of images of $U_i$.

        (We recall that in the analytic topology, $f\colon U\to X (\CC)$ is \textbf{\'{e}tale}
        if it is a \emph{local on the source homeomorphism}: for each $u \in U$ there
        exists an open neighborhood $u \ni V$ such that
        $\left.f\right|_V\colon V\to f(V)$ is a homeomorphism.)

        Since the inclusion of an open subset $U \subset X (\CC)$ is an \'{e}tale map,
        we have a fully faithful functor $X (\CC) \subset X_\text{\it cl}$, and the
        topology on $X (\CC)$ is induced by the topology on $X_\text{\it cl}$. This
        gives us a morphism of sites $\delta\colon X_\text{\it cl} \to X (\CC)$, which
        by the comparison lemma \cite[Expos\'{e}~III, Th\'{e}or\`{e}me~4.1]{SGA4}
        induces an equivalence of the corresponding categories of sheaves
        $$\delta_*\colon \mathbf{Sh} (X_\text{\it cl}) \to \mathbf{Sh} (X (\CC)).$$

  \item A morphism of schemes $f\colon X'_\CC \to X_\CC$ over $\Spec \CC$ is \'{e}tale
        if and only if the map $f (\CC)\colon X' (\CC) \to X (\CC)$ is \'{e}tale
        \cite[Expos\'{e}~XII, Proposition~3.1]{SGA1}, and therefore the functor $X'_\CC
          \rightsquigarrow X' (\CC)$ gives us a morphism of sites $$\epsilon\colon
          X_\text{\it cl} \to X_{\CC,\text{\it \'{e}t}}.$$
\end{enumerate}

\begin{definition}
  We define the functor
  $$\alpha^*\colon \mathbf{Sh} (X_\et) \to \mathbf{Sh} (G_\RR, X (\CC))$$
  via the composition
  \[ \begin{tikzcd}
      \mathbf{Sh} (X_\et) \ar{r}{\gamma^*} &
      \mathbf{Sh} (X_{\CC,\text{\it \'{e}t}}) \ar{r}{\epsilon^*} &
      \mathbf{Sh} (X_\text{\it cl}) \ar{r}{\delta_*}[swap]{\simeq} &
      \mathbf{Sh} (X (\CC))
    \end{tikzcd} \]
\end{definition}

As we start from a scheme over $\Spec \ZZ$ and base change to $\Spec \CC$, the
resulting sheaf on $X (\CC)$ is equivariant with respect to the complex
conjugation, hence an object in $\mathbf{Sh} (G_\RR, X (\CC))$. For the
definition of equivariant sheaves, we refer to the introduction.

\begin{lemma}
  \label{lemma:alpha-preserves-colimits}
  $\alpha^*$ preserves colimits.

  \begin{proof}
    $\alpha^*$ is the composition of the inverse image functors $\gamma^*$ and
    $\epsilon^*$ (which are left adjoint) and an equivalence $\delta_*$.
  \end{proof}
\end{lemma}

\begin{proposition}
  \label{prop:inverse-image-gamma}
  Given a sheaf $\mathcal{F}$ on $X_\et$, there exists a natural morphism
  $$\Gamma (X_\et, \mathcal{F}) \to \Gamma (G_\RR, X (\CC), \alpha^* \mathcal{F}),$$
  and similarly, for cohomology with compact support,
  $$\Gamma_c (X_\et, \mathcal{F}) \to \Gamma_c (G_\RR, X (\CC), \alpha^* \mathcal{F}).$$

  \begin{proof}
    If $j\colon X \hookrightarrow \mathfrak{X}$ is a Nagata compactification, we
    have the corresponding compactification
    $j (\CC)\colon X (\CC) \hookrightarrow \mathfrak{X} (\CC)$. The extension by
    zero morphism
    $j (\CC)_!\colon \mathbf{Sh} (X (\CC)) \to \mathbf{Sh} (\mathfrak{X} (\CC))$
    restricts to the subcategory of $G_\RR$-equivariant sheaves: if
    $\mathcal{F}$ is a $G_\RR$-equivariant sheaf on $X (\CC)$, then
    $j (\CC)_!  \mathcal{F}$ is a $G_\RR$-equivariant sheaf on
    $\mathfrak{X} (\CC)$. From the definition of $\alpha^*$, we see that that
    there is a commutative diagram
    \[ \begin{tikzcd}
        \mathbf{Sh} (X_\et) \ar{r}{\alpha^*}\ar{d}[swap]{j_!} & \mathbf{Sh} (G_\RR, X (\CC)) \ar{d}{j (\CC)_!} \\
        \mathbf{Sh} (\mathfrak{X}_\et) \ar{r}[swap]{\alpha^*_\mathfrak{X}} & \mathbf{Sh} (G_\RR, \mathfrak{X} (\CC))
      \end{tikzcd} \]
    ---this diagram commutes for representable \'{e}tale sheaves, and then every
    \'{e}tale sheaf is a colimit of representable sheaves, and $\alpha^*$,
    $j_!$, $\alpha^*_\mathfrak{X}$, $j (\CC)_!$ preserve colimits, as left
    adjoints.

    The morphism in question is given by
    \begin{multline*}
      \Gamma_c (X_\et, \mathcal{F}) \dfn \Gamma (\mathfrak{X}_\et, j_! \mathcal{F}) \to
      \Gamma (G_\RR, \mathfrak{X} (\CC), \alpha^*_\mathfrak{X} j_! \mathcal{F}) \\
      =
      \Gamma (G_\RR, \mathfrak{X} (\CC), j (\CC)_! \, \alpha^* \mathcal{F})
      \rdfn \Gamma_c (G_\RR, X (\CC), \alpha^* \mathcal{F}). \qedhere
    \end{multline*}
  \end{proof}
\end{proposition}

The morphism $\alpha$ is also discussed in \cite[Appendix~A]{Flach-Morin-2018},
but Flach and Morin work with proper schemes; the above remarks are to make
sure that everything works fine for compactifications.

\subsection*{Modified \'{e}tale cohomology}

Here we briefly review the \textbf{modified \'{e}tale cohomology with compact
support} $R\widehat{\Gamma}_c (X_\et, -)$. It was introduced by Th.~Zink in
\cite[Appendix~2]{Haberland-1978} for the case of number rings $X = \Spec
  \mathcal{O}_{K,S}$, and it is also discussed in \cite[\S II.2]{Milne-ADT}. The
general definition for $X \to \Spec\ZZ$ is treated in \cite[\S
  6.7]{Flach-Morin-2018} and \cite[\S 2]{Geisser-Schmidt-2018}.

Thanks to the Leray spectral sequence $R\Gamma (\mathfrak{X}_\et, -) \cong
  R\Gamma (\Spec \ZZ_\et, -)\circ R g_*$, we have
\[
  R\Gamma_c (X_\et, \mathcal{F}) \dfn R\Gamma (\mathfrak{X}_\et, j_!\mathcal{F})
  \cong R\Gamma ((\Spec \ZZ)_\et, R f_! \mathcal{F}), \quad
  \text{where } Rf_! \mathcal{F} \dfn R g_* j_! \mathcal{F}.
\]

First we recall that for a finite group $G$ and a $G$-module $A$ the
corresponding group cohomology $H^i (G,A)$ (resp. Tate cohomology
$\widehat{H}^i (G,A)$) can be defined in terms of resolutions $P_\bullet$
(resp. complete resolutions $\widehat{P}_\bullet$) of $\ZZ$ by free $\ZZ
  G$-modules (see e.g. \cite[Chapter~VI]{Brown-1994}). More generally, if
$A^\bullet$ is a bounded (cohomological) complex of $G$-modules, we obtain a
\emph{double complex} of abelian groups $\Hom^{\bullet\bullet} (P_\bullet,
  A^\bullet)$ (resp. $\Hom^{\bullet\bullet} (\widehat{P}_\bullet, A^\bullet)$),
and it makes sense to define the corresponding \textbf{group hypercohomology}
(resp. \textbf{Tate hypercohomology}) via the complexes
\[ R\Gamma (G, A^\bullet) \dfn
  \Tot^\oplus (\Hom^{\bullet\bullet} (P_\bullet, A^\bullet)), \quad
  R\widehat{\Gamma} (G, A^\bullet) \dfn
  \Tot^\oplus (\Hom^{\bullet\bullet} (\widehat{P}_\bullet, A^\bullet)). \]

Now if $\mathcal{F}$ is an abelian sheaf on $(\Spec \ZZ)_\et$, then the
corresponding \textbf{modified cohomology with compact support} is
characterized by the distinguished triangle
\[ R\widehat{\Gamma}_c ((\Spec \ZZ)_\et, \mathcal{F}) \to
  R\Gamma ((\Spec \ZZ)_\et, \mathcal{F}) \to
  R\widehat{\Gamma} (G_\RR, v^* \mathcal{F}) \to
  R\widehat{\Gamma}_c ((\Spec \ZZ)_\et, \mathcal{F}) [1] \]
Here $v\colon \Spec \RR \to \Spec \ZZ$ is the canonical morphism, and $v^*
  \mathcal{F}$ is the corresponding sheaf on $(\Spec \RR)_\et$, which can be
viewed as a $G_\RR$-module by \cite[Expos\'{e}~VII, 2.3]{SGA4}, and
$R\widehat{\Gamma} (G_\RR, v^* \mathcal{F})$ denotes the corresponding Tate
cohomology.

In general, given an arithmetic scheme $X \to \Spec \ZZ$ and a torsion abelian
sheaf $\mathcal{F}$ on $X_\et$, we choose a Nagata compactification as above
and set
\[ R\widehat{\Gamma}_c (X_\et, \mathcal{F}) \dfn
  R\widehat{\Gamma}_c ((\Spec \ZZ)_\et, R f_! \mathcal{F}). \]
We have a natural morphism $$R\widehat{\Gamma}_c (X_\et, \mathcal{F}) \to
  R\Gamma_c (X_\et, \mathcal{F}),$$ which is an isomorphism if $X (\RR) =
  \emptyset$. In general, Tate cohomology $\widehat{H}^i (G_\RR, -)$ is
annihilated by multiplication by $2 = \# G_\RR$, and therefore $\widehat{H}^i_c
  (X_\et,\mathcal{F}) \to H^i_c (X_\et,\mathcal{F})$ has $2$-torsion kernel and
cokernel.

For canonicity and functoriality, I refer to \cite[\S 2]{Geisser-Schmidt-2018}.


\bigskip


\address{
  Alexey Beshenov\\
  Cosmos 2674\\
  Col. Jardines del Bosque\\
  Guadalajara, Jalisco\\
  Mexico C.P. 44520}
{cadadr@gmail.com}

\end{document}